\documentclass
[numbers=enddot,12pt,final,onecolumn,notitlepage,abstracton]{scrartcl}%
\usepackage[headsepline,footsepline,manualmark]{scrlayer-scrpage}
\usepackage[all,cmtip]{xy}
\usepackage{amssymb}
\usepackage{amsmath}
\usepackage{amsthm}
\usepackage{framed}
\usepackage{comment}
\usepackage{color}
\usepackage{tabu}
\usepackage[sc]{mathpazo}
\usepackage[T1]{fontenc}
\usepackage{needspace}
\usepackage[breaklinks]{hyperref}
\providecommand{\U}[1]{\protect\rule{.1in}{.1in}}
\newcounter{exer}
\theoremstyle{definition}
\newtheorem{theo}{Theorem}[section]
\newenvironment{theorem}[1][]
{\begin{theo}[#1]\begin{leftbar}}
{\end{leftbar}\end{theo}}
\newtheorem{lem}[theo]{Lemma}
\newenvironment{lemma}[1][]
{\begin{lem}[#1]\begin{leftbar}}
{\end{leftbar}\end{lem}}
\newtheorem{prop}[theo]{Proposition}
\newenvironment{proposition}[1][]
{\begin{prop}[#1]\begin{leftbar}}
{\end{leftbar}\end{prop}}
\newtheorem{defi}[theo]{Definition}
\newenvironment{definition}[1][]
{\begin{defi}[#1]\begin{leftbar}}
{\end{leftbar}\end{defi}}
\newtheorem{remk}[theo]{Remark}
\newenvironment{remark}[1][]
{\begin{remk}[#1]\begin{leftbar}}
{\end{leftbar}\end{remk}}
\newtheorem{coro}[theo]{Corollary}
\newenvironment{corollary}[1][]
{\begin{coro}[#1]\begin{leftbar}}
{\end{leftbar}\end{coro}}
\newtheorem{conv}[theo]{Convention}

\newtheorem{quest}[theo]{Question}

\newtheorem{warn}[theo]{Warning}

\newtheorem{soln}{Solution}

\newtheorem{conj}[theo]{Conjecture}

\newtheorem{exam}[theo]{Example}
\newenvironment{example}[1][]
{\begin{exam}[#1]\begin{leftbar}}
{\end{leftbar}\end{exam}}
\newtheorem{exmp}[exer]{Exercise}

\let\sumnonlimits\sum
\let\prodnonlimits\prod
\renewcommand{\sum}{\sumnonlimits\limits}
\renewcommand{\prod}{\prodnonlimits\limits}
\newenvironment{verlong}{}{}

\excludecomment{verlong}
\includecomment{vershort}
\excludecomment{noncompile}

\ihead{A generalization of Chio Pivotal Condensation}
\ohead{page \thepage}
\begin{document}

\author{Darij Grinberg, Karthik Karnik, Anya Zhang}
\title{From Chio Pivotal Condensation to the Matrix-Tree theorem}
\date{\today}
\maketitle

\begin{abstract}
We show a determinant identity which generalizes both the Chio pivotal
condensation theorem and the Matrix-Tree theorem.

\end{abstract}

\section{Introduction}

The Chio pivotal condensation theorem (Theorem \ref{thm.chio} below, or
\cite[Theorem 3.6.1]{Eves}) is a simple particular case of the Dodgson-Muir
determinantal identity (\cite[(4)]{BerBru08}), which can be used to reduce the
computation of an $n\times n$-determinant to that of an $\left(  n-1\right)
\times\left(  n-1\right)  $-determinant (provided that an entry of the matrix
can be divided by\footnote{We work with matrices over arbitrary commutative
rings, so this is not a moot point. Of course, if the ring is a field, then
this just means that the matrix has a nonzero entry.}). On the other hand, the
Matrix-Tree theorem (Theorem \ref{thm.mtt}, or \cite[Section 4]{Zeilbe}, or
\cite[Theorem 1]{Verstraete}) expresses the number of spanning trees of a
graph as a determinant\footnote{And not just the number; rather, a
\textquotedblleft weighted number\textquotedblright\ from which the spanning
trees can be read off if the weights are chosen generically enough.}. In this
note, we show that these two results have a common generalization (Theorem
\ref{thm.supergen}). As we have tried to keep the note self-contained, using
only the well-known fundamental properties of determinants, it also provides
new proofs for both results.

\subsection{Acknowledgments}

We thank the PRIMES project at MIT, during whose 2015 iteration this paper was
created, and in particular George Lusztig for sponsoring the first author's
mentorship in this project.

\section{The theorems}

We shall use the (rather standard) notations defined in \cite{detnotes}. In
particular, $\mathbb{N}$ means the set $\left\{  0,1,2,\ldots\right\}  $. For
any $n\in\mathbb{N}$, we let $S_{n}$ denote the group of permutations of the
set $\left\{  1,2,\ldots,n\right\}  $. The $n\times m$-matrix whose $\left(
i,j\right)  $-th entry is $a_{i,j}$ for each $\left(  i,j\right)  \in\left\{
1,2,\ldots,n\right\}  \times\left\{  1,2,\ldots,m\right\}  $ will be denoted
by $\left(  a_{i,j}\right)  _{1\leq i\leq n,\ 1\leq j\leq m}$.

Let $\mathbb{K}$ be a commutative ring. We shall regard $\mathbb{K}$ as fixed
throughout this note (so we won't always write \textquotedblleft Let
$\mathbb{K}$ be a commutative ring\textquotedblright\ in our propositions);
the notion \textquotedblleft matrix\textquotedblright\ will always mean
\textquotedblleft matrix with entries in $\mathbb{K}$\textquotedblright.

\subsection{Chio Pivotal Condensation}

We begin with a statement of the Chio Pivotal Condensation theorem (see, e.g.,
\cite[Theorem 0.1]{KarZha16} and the reference therein):

\begin{theorem}
\label{thm.chio}Let $n\geq2$ be an integer. Let $A=\left(  a_{i,j}\right)
_{1\leq i\leq n,\ 1\leq j\leq n}\in\mathbb{K}^{n\times n}$ be a matrix. Then,
\[
\det\left(  \left(  a_{i,j}a_{n,n}-a_{i,n}a_{n,j}\right)  _{1\leq i\leq
n-1,\ 1\leq j\leq n-1}\right)  =a_{n,n}^{n-2}\cdot\det\left(  \left(
a_{i,j}\right)  _{1\leq i\leq n,\ 1\leq j\leq n}\right)  .
\]

\end{theorem}

\begin{example}
If $n=3$ and $A=\left(
\begin{array}
[c]{ccc}%
a & a^{\prime} & a^{\prime\prime}\\
b & b^{\prime} & b^{\prime\prime}\\
c & c^{\prime} & c^{\prime\prime}%
\end{array}
\right)  $, then Theorem \ref{thm.chio} says that%
\[
\det\left(
\begin{array}
[c]{cc}%
ac^{\prime\prime}-a^{\prime\prime}c & a^{\prime}c^{\prime\prime}%
-a^{\prime\prime}c^{\prime}\\
bc^{\prime\prime}-b^{\prime\prime}c & b^{\prime}c^{\prime\prime}%
-b^{\prime\prime}c^{\prime}%
\end{array}
\right)  =\left(  c^{\prime\prime}\right)  ^{3-2}\cdot\det\left(
\begin{array}
[c]{ccc}%
a & a^{\prime} & a^{\prime\prime}\\
b & b^{\prime} & b^{\prime\prime}\\
c & c^{\prime} & c^{\prime\prime}%
\end{array}
\right)  .
\]

\end{example}

Theorem \ref{thm.chio} (originally due to F\'{e}lix Chio in 1853\footnote{See
\cite[footnote 2]{Heinig} and \cite[\S 2]{Abeles} for some historical
background.}) is nowadays usually regarded either as a particular case of the
Dodgson-Muir determinantal identity (\cite[(4)]{BerBru08}), or as a relatively
easy exercise on row operations and the method of universal
identities\footnote{In more detail:
\par
\begin{itemize}
\item In order to derive Theorem \ref{thm.chio} from \cite[(4)]{BerBru08}, it
suffices to set $k=n-1$ and recognize the right hand side of \cite[(4)]%
{BerBru08} as $\det\left(  \left(  a_{i,j}a_{n,n}-a_{i,n}a_{n,j}\right)
_{1\leq i\leq n-1,\ 1\leq j\leq n-1}\right)  $.
\par
\item A proof of Theorem \ref{thm.chio} using row operations can be found in
\cite[Theorem 3.6.1]{Eves}, up to a few minor issues: First of all,
\cite[Theorem 3.6.1]{Eves} proves not exactly Theorem \ref{thm.chio} but the
analogous identity
\[
\det\left(  \left(  a_{i+1,j+1}a_{1,1}-a_{i+1,1}a_{1,j+1}\right)  _{1\leq
i\leq n-1,\ 1\leq j\leq n-1}\right)  =a_{1,1}^{n-2}\cdot\det\left(  \left(
a_{i,j}\right)  _{1\leq i\leq n,\ 1\leq j\leq n}\right)  .
\]
Second, \cite[Theorem 3.6.1]{Eves} assumes $a_{1,1}$ to be invertible (and all
$a_{i,j}$ to belong to a field); however, assumptions like this can easily be
disposed of using the method of universal identities (see \cite{Conrad09}).
\end{itemize}
\par
A more explicit and self-contained proof of Theorem \ref{thm.chio} can be
found in \cite{KarZha16}. References to other proofs appear in \cite[\S 2]%
{Abeles}.}. We, however, shall generalize it in a different direction.

\subsection{Generalization, step 1}

Our generalization will proceed in two steps. In the first step, we shall
replace some of the $n$'s on the left hand side by $f\left(  i\right)  $'s
(see Theorem \ref{thm.chio-gen} below). We first define some notations:

\begin{definition}
Let $n$ be a positive integer. Let $f:\left\{  1,2,\ldots,n\right\}
\rightarrow\left\{  1,2,\ldots,n\right\}  $ be any map such that $f\left(
n\right)  =n$.

We say that the map $f$ is $n$\textit{-potent} if for every $i\in\left\{
1,2,\ldots,n\right\}  $, there exists some $k\in\mathbb{N}$ such that
$f^{k}\left(  i\right)  =n$. (In less formal terms, $f$ is $n$-potent if and
only if every element of $\left\{  1,2,\ldots,n\right\}  $ eventually arrives
at $n$ when being subjected to repeated application of $f$.)
\end{definition}

(Note that, by definition, any $n$-potent map $f:\left\{  1,2,\ldots
,n\right\}  \rightarrow\left\{  1,2,\ldots,n\right\}  $ must satisfy $f\left(
n\right)  =n$.)

\begin{example}
For this example, let $n=3$. The map $\left\{  1,2,3\right\}  \rightarrow
\left\{  1,2,3\right\}  $ sending $1,2,3$ to $2,1,3$, respectively, is not
$n$-potent (because applying it repeatedly to $1$ can only give $1$ or $2$,
but never $3$). The map $\left\{  1,2,3\right\}  \rightarrow\left\{
1,2,3\right\}  $ sending $1,2,3$ to $3,3,2$, respectively, is not $n$-potent
(since it does not send $n$ to $n$). The map $\left\{  1,2,3\right\}
\rightarrow\left\{  1,2,3\right\}  $ sending $1,2,3$ to $3,1,3$, respectively,
is $n$-potent (indeed, every element of $\left\{  1,2,3\right\}  $ goes to $3$
after at most two applications of this map).
\end{example}

\begin{remark}
\label{rmk.n-potent.trees}Given a positive integer $n$, the $n$-potent maps
$f:\left\{  1,2,\ldots,n\right\}  \rightarrow\left\{  1,2,\ldots,n\right\}  $
are in 1-to-1 correspondence with the trees with vertex set $\left\{
1,2,\ldots,n\right\}  $. Namely, an $n$-potent map $f$ corresponds to the tree
whose edges are $\left\{  i,f\left(  i\right)  \right\}  $ for all
$i\in\left\{  1,2,\ldots,n-1\right\}  $. If we regard the tree as a rooted
tree with root $n$, and if we direct every edge towards the root, then the
edges are $\left(  i,f\left(  i\right)  \right)  $ for all $i\in\left\{
1,2,\ldots,n-1\right\}  $.
\end{remark}

\begin{remark}
\label{rmk.n-potent.atleast1}Let $n\geq2$ be an integer. Let $f:\left\{
1,2,\ldots,n\right\}  \rightarrow\left\{  1,2,\ldots,n\right\}  $ be any
$n$-potent map. Then:

\textbf{(a)} There exists some $g\in\left\{  1,2,\ldots,n-1\right\}  $ such
that $f\left(  g\right)  =n$.

\textbf{(b)} We have $\left\vert f^{-1}\left(  n\right)  \right\vert \geq2$.
\end{remark}

The (very simple) proof of Remark \ref{rmk.n-potent.atleast1} can be found in
the Appendix (Section \ref{sect.app}).

\begin{definition}
\label{def.n-potent.weightabut}Let $n\geq2$ be an integer. Let $A=\left(
a_{i,j}\right)  _{1\leq i\leq n,\ 1\leq j\leq n}\in\mathbb{K}^{n\times n}$ be
an $n\times n$-matrix. Let $f:\left\{  1,2,\ldots,n\right\}  \rightarrow
\left\{  1,2,\ldots,n\right\}  $ be any $n$-potent map.

\textbf{(a)} We define an element $\operatorname*{weight}\nolimits_{f}A$ of
$\mathbb{K}$ by%
\[
\operatorname*{weight}\nolimits_{f}A=\prod_{i=1}^{n-1}a_{i,f\left(  i\right)
}.
\]

\textbf{(b)} We define an element $\operatorname*{abut}\nolimits_{f}A$ of
$\mathbb{K}$ by%
\[
\operatorname*{abut}\nolimits_{f}A=a_{n,n}^{\left\vert f^{-1}\left(  n\right)
\right\vert -2}\prod_{\substack{i\in\left\{  1,2,\ldots,n-1\right\}
;\\f\left(  i\right)  \neq n}}a_{f\left(  i\right)  ,n}.
\]
(This is well-defined, since Remark \ref{rmk.n-potent.atleast1} \textbf{(b)}
shows that $\left\vert f^{-1}\left(  n\right)  \right\vert -2\in\mathbb{N}$.)
\end{definition}

\begin{remark}
\label{rmk.n-potent.abut}Let $n$, $A$ and $f$ be as in Definition
\ref{def.n-potent.weightabut}. Here are two slightly more intuitive ways to
think of $\operatorname*{abut}\nolimits_{f}A$:

\textbf{(a)} If $a_{n,n}\in\mathbb{K}$ is invertible, then
$\operatorname*{abut}\nolimits_{f}A$ is simply $\dfrac{1}{a_{n,n}}\prod
_{i\in\left\{  1,2,\ldots,n-1\right\}  }a_{f\left(  i\right)  ,n}$.

\textbf{(b)} Remark \ref{rmk.n-potent.atleast1} \textbf{(a)} shows that there
exists some $g\in\left\{  1,2,\ldots,n-1\right\}  $ such that $f\left(
g\right)  =n$. Fix such a $g$. Then,%
\[
\operatorname*{abut}\nolimits_{f}A=\prod_{\substack{i\in\left\{
1,2,\ldots,n-1\right\}  ;\\i\neq g}}a_{f\left(  i\right)  ,n}.
\]

\end{remark}

The (nearly trivial) proof of Remark \ref{rmk.n-potent.abut} is again found in
the Appendix.

Now, we can state our first generalization of Theorem \ref{thm.chio}:

\begin{theorem}
\label{thm.chio-gen}Let $n$ be a positive integer. Let $A=\left(
a_{i,j}\right)  _{1\leq i\leq n,\ 1\leq j\leq n}\in\mathbb{K}^{n\times n}$ be
an $n\times n$-matrix. Let $f:\left\{  1,2,\ldots,n\right\}  \rightarrow
\left\{  1,2,\ldots,n\right\}  $ be any map such that $f\left(  n\right)  =n$.

Let $B$ be the $\left(  n-1\right)  \times\left(  n-1\right)  $-matrix%
\[
\left(  a_{i,j}a_{f\left(  i\right)  ,n}-a_{i,n}a_{f\left(  i\right)
,j}\right)  _{1\leq i\leq n-1,\ 1\leq j\leq n-1}\in\mathbb{K}^{\left(
n-1\right)  \times\left(  n-1\right)  }.
\]
{}

\textbf{(a)} If the map $f$ is not $n$-potent, then $\det B=0$.

\textbf{(b)} Assume that $n\geq2$. Assume that the map $f$ is $n$-potent.
Then,%
\[
\det B=\left(  \operatorname*{abut}\nolimits_{f}A\right)  \cdot\det A.
\]

\end{theorem}

\begin{example}
For this example, let $n=3$ and $A=\left(
\begin{array}
[c]{ccc}%
a_{1,1} & a_{1,2} & a_{1,3}\\
a_{2,1} & a_{2,2} & a_{2,3}\\
a_{3,1} & a_{3,2} & a_{3,3}%
\end{array}
\right)  $.

If $f:\left\{  1,2,3\right\}  \rightarrow\left\{  1,2,3\right\}  $ is the map
sending $1,2,3$ to $3,1,3$, respectively, then the matrix $B$ defined in
Theorem \ref{thm.chio-gen} is $\left(
\begin{array}
[c]{cc}%
a_{1,1}a_{3,3}-a_{1,3}a_{3,1} & a_{1,2}a_{3,3}-a_{1,3}a_{3,2}\\
a_{2,1}a_{1,3}-a_{2,3}a_{1,1} & a_{2,2}a_{1,3}-a_{2,3}a_{1,2}%
\end{array}
\right)  $. Since this map $f$ is $n$-potent, Theorem \ref{thm.chio-gen}
\textbf{(b)} predicts that this matrix $B$ satisfies $\det B=\left(
\operatorname*{abut}\nolimits_{f}A\right)  \cdot\det A$. This is indeed easily
checked (indeed, we have $\operatorname*{abut}\nolimits_{f}A=a_{1,3}$ in this case).

On the other hand, if $f:\left\{  1,2,3\right\}  \rightarrow\left\{
1,2,3\right\}  $ is the map sending $1,2,3$ to $1,1,3$, respectively, then the
matrix $B$ defined in Theorem \ref{thm.chio-gen} is $\left(
\begin{array}
[c]{cc}%
a_{1,1}a_{1,3}-a_{1,3}a_{1,1} & a_{1,2}a_{1,3}-a_{1,3}a_{1,2}\\
a_{2,1}a_{1,3}-a_{2,3}a_{1,1} & a_{2,2}a_{1,3}-a_{2,3}a_{1,2}%
\end{array}
\right)  $. Since this map $f$ is not $n$-potent, Theorem \ref{thm.chio-gen}
\textbf{(a)} predicts that this matrix $B$ satisfies $\det B=0$. This, too, is
easily checked (and arguably obvious in this case).
\end{example}

Applying Theorem \ref{thm.chio-gen} \textbf{(b)} to $f\left(  i\right)  =n$
yields Theorem \ref{thm.chio}. (The map $f:\left\{  1,2,\ldots,n\right\}
\rightarrow\left\{  1,2,\ldots,n\right\}  $ defined by $f\left(  i\right)  =n$
is clearly $n$-potent, and satisfies $\operatorname*{abut}\nolimits_{f}%
A=a_{n,n}^{n-2}$.)

We defer the proof of Theorem \ref{thm.chio-gen} until later; first, let us
see how it can be generalized a bit further (not substantially, anymore) and
how this generalization also encompasses the matrix-tree theorem.

\subsection{The matrix-tree theorem}

\begin{definition}
For any two objects $i$ and $j$, we define an element $\delta_{i,j}%
\in\mathbb{K}$ by $\delta_{i,j}=%
\begin{cases}
1, & \text{if }i=j;\\
0, & \text{if }i\neq j
\end{cases}
$.
\end{definition}

Let us first state the matrix-tree theorem.

To be honest, there is no \textquotedblleft the matrix-tree
theorem\textquotedblright, but rather a network of \textquotedblleft
matrix-tree theorems\textquotedblright\ (some less, some more general), each
of which has a reasonable claim to this name. Here we shall prove the
following one:

\begin{theorem}
\label{thm.mtt}Let $n\geq1$ be an integer. Let $W:\left\{  1,2,\ldots
,n\right\}  \times\left\{  1,2,\ldots,n\right\}  \rightarrow\mathbb{K}$ be any
function. For every $i\in\left\{  1,2,\ldots,n\right\}  $, set%
\[
d^{+}\left(  i\right)  =\sum_{j=1}^{n}W\left(  i,j\right)  .
\]
Let $L$ be the matrix $\left(  \delta_{i,j}d^{+}\left(  i\right)  -W\left(
i,j\right)  \right)  _{1\leq i\leq n-1,\ 1\leq j\leq n-1}\in\mathbb{K}%
^{\left(  n-1\right)  \times\left(  n-1\right)  }$. Then,%
\begin{equation}
\det L=\sum_{\substack{f:\left\{  1,2,\ldots,n\right\}  \rightarrow\left\{
1,2,\ldots,n\right\}  ;\\f\left(  n\right)  =n;\\f\text{ is }n\text{-potent}%
}}\prod_{i=1}^{n-1}W\left(  i,f\left(  i\right)  \right)  .
\label{eq.thm.mtt.detL=}%
\end{equation}

\end{theorem}

Since our notation differs from that in most other sources on the matrix-tree
theorem, let us explain the equivalence between our Theorem \ref{thm.mtt} and
one of its better-known avatars: The version of the matrix-tree theorem stated
in \cite[Section 4]{Zeilbe} involves some \textquotedblleft
weights\textquotedblright\ $a_{k,m}$, a determinant of an $\left(  n-1\right)
\times\left(  n-1\right)  $-matrix, and a sum over a set $\mathcal{T}%
=\mathcal{T}\left(  n\right)  $. These correspond (respectively) to the values
$W\left(  k,m\right)  $, the determinant $\det L$, and the sum over all
$n$-potent maps $f$ in our Theorem \ref{thm.mtt}. In fact, the only nontrivial
part of this correspondence is the bijection between the trees in
$\mathcal{T}$ and the $n$-potent maps $f$ over which the sum in
(\ref{eq.thm.mtt.detL=}) ranges. This bijection is precisely the one
introduced in Remark \ref{rmk.n-potent.trees}.\footnote{A slightly different
version of the matrix-tree theorem appears in \cite[Theorem 1]{Verstraete}
(and various other places); it involves a function $W$, a number $v\in\left\{
1,2,\ldots,n\right\}  $, a matrix $L_{v}$, a set $\mathcal{T}_{v}$ and a sum
$\tau\left(  W,v\right)  $. Our Theorem \ref{thm.mtt} is equivalent to the
case of \cite[Theorem 1]{Verstraete} for $v=n$; but this case is easily seen
to be equivalent to the general case of \cite[Theorem 1]{Verstraete} (since
the elements of $\left\{  1,2,\ldots,n\right\}  $ can be permuted at will).
Our matrix $L$ is the $L_{n}$ of \cite[Theorem 1]{Verstraete}. Furthermore,
our sum over all $n$-potent maps $f$ corresponds to the sum $\tau\left(
W,n\right)  $ in \cite{Verstraete}, which is a sum over all $n$-arborescences
on $\left\{  1,2,\ldots,n\right\}  $; the correspondence is again due to
Remark \ref{rmk.n-potent.trees}.}

It might seem weird to call Theorem \ref{thm.mtt} the \textquotedblleft
matrix-tree theorem\textquotedblright\ if the word \textquotedblleft
tree\textquotedblright\ never occurs inside it. However, as we have already
noticed in Remark \ref{rmk.n-potent.trees}, the trees on the set $\left\{
1,2,\ldots,n\right\}  $ are in bijection with the $n$-potent maps $\left\{
1,2,\ldots,n\right\}  \rightarrow\left\{  1,2,\ldots,n\right\}  $, and
therefore the sum on the right hand side of (\ref{eq.thm.mtt.detL=}) can be
viewed as a sum over all these trees. Moreover, the function $W$ can be viewed
as an $n\times n$-matrix; when this matrix is specialized to the adjacency
matrix of a directed graph, the sum on the right hand side of
(\ref{eq.thm.mtt.detL=}) becomes the number of directed spanning trees of this
directed graph directed towards the root $n$.

\subsection{Generalization, step 2}

Now, as promised, we will generalize Theorem \ref{thm.chio-gen} a step
further. While the result will not be significantly stronger (we will actually
derive it from Theorem \ref{thm.chio-gen} quite easily), it will lead to a
short proof of Theorem \ref{thm.mtt}:

\Needspace{8cm}

\begin{theorem}
\label{thm.supergen}Let $n\geq2$ be an integer. Let $A=\left(  a_{i,j}\right)
_{1\leq i\leq n,\ 1\leq j\leq n}\in\mathbb{K}^{n\times n}$ and $B=\left(
b_{i,j}\right)  _{1\leq i\leq n,\ 1\leq j\leq n}\in\mathbb{K}^{n\times n}$ be
$n\times n$-matrices. Write the $n\times n$-matrix $BA$ in the form
$BA=\left(  c_{i,j}\right)  _{1\leq i\leq n,\ 1\leq j\leq n}$.

Let $G$ be the $\left(  n-1\right)  \times\left(  n-1\right)  $-matrix%
\[
\left(  a_{i,j}c_{i,n}-a_{i,n}c_{i,j}\right)  _{1\leq i\leq n-1,\ 1\leq j\leq
n-1}\in\mathbb{K}^{\left(  n-1\right)  \times\left(  n-1\right)  }.
\]
Then,%
\[
\det G=\left(  \sum_{\substack{f:\left\{  1,2,\ldots,n\right\}  \rightarrow
\left\{  1,2,\ldots,n\right\}  ;\\f\left(  n\right)  =n;\\f\text{ is
}n\text{-potent}}}\left(  \operatorname*{weight}\nolimits_{f}B\right)  \left(
\operatorname*{abut}\nolimits_{f}A\right)  \right)  \cdot\det A.
\]

\end{theorem}

To obtain Theorem \ref{thm.chio-gen} from Theorem \ref{thm.supergen}, we have
to define $B$ by $B=\left(  \delta_{j,f\left(  i\right)  }\right)  _{1\leq
i\leq n,\ 1\leq j\leq n}$. Below we shall show how to obtain the matrix-tree
theorem from Theorem \ref{thm.supergen}.

\begin{example}
Let us see what Theorem \ref{thm.supergen} says for $n=3$. There are three
$n$-potent maps $f:\left\{  1,2,3\right\}  \rightarrow\left\{  1,2,3\right\}
$:

\begin{itemize}
\item one map $f_{33}$ which sends both $1$ and $2$ to $3$;

\item one map $f_{23}$ which sends $1$ to $2$ and $2$ to $3$;

\item one map $f_{31}$ which sends $2$ to $1$ and $1$ to $3$.
\end{itemize}

The definition of the $c_{i,j}$ as the entries of $BA$ shows that
$c_{i,j}=b_{i,1}a_{1,j}+b_{i,2}a_{2,j}+b_{i,3}a_{3,j}$ for all $i$ and $j$. We
have%
\[
G=\left(
\begin{array}
[c]{cc}%
a_{1,1}c_{1,3}-c_{1,1}a_{1,3} & a_{1,2}c_{1,3}-c_{1,2}a_{1,3}\\
a_{2,1}c_{2,3}-c_{2,1}a_{2,3} & a_{2,2}c_{2,3}-c_{2,2}a_{2,3}%
\end{array}
\right)  .
\]
Theorem \ref{thm.supergen} says that%
\begin{align*}
\det G  &  =\left(  \left(  \operatorname*{weight}\nolimits_{f_{33}}B\right)
\left(  \operatorname*{abut}\nolimits_{f_{33}}A\right)  +\left(
\operatorname*{weight}\nolimits_{f_{23}}B\right)  \left(  \operatorname*{abut}%
\nolimits_{f_{23}}A\right)  \right. \\
&  \ \ \ \ \ \ \ \ \ \ \left.  +\left(  \operatorname*{weight}%
\nolimits_{f_{31}}B\right)  \left(  \operatorname*{abut}\nolimits_{f_{31}%
}A\right)  \right)  \cdot\det A\\
&  =\left(  b_{1,3}b_{2,3}a_{3,3}+b_{1,2}b_{2,3}a_{2,3}+b_{1,3}b_{2,1}%
a_{1,3}\right)  \cdot\det A.
\end{align*}

\end{example}

\section{The proofs}

\subsection{Deriving Theorem \ref{thm.supergen} from Theorem
\ref{thm.chio-gen}}

Let us see how Theorem \ref{thm.supergen} can be proven using Theorem
\ref{thm.chio-gen} (which we have not proven yet). We shall need two lemmas:

\begin{lemma}
\label{lem.supergen.sum1}Let $n\in\mathbb{N}$ and $m\in\mathbb{N}$. Let
$b_{i,k}$ be an element of $\mathbb{K}$ for every $i\in\left\{  1,2,\ldots
,m\right\}  $ and every $k\in\left\{  1,2,\ldots,n\right\}  $. Let $d_{i,j,k}$
be an element of $\mathbb{K}$ for every $i\in\left\{  1,2,\ldots,m\right\}  $,
$j\in\left\{  1,2,\ldots,m\right\}  $ and $k\in\left\{  1,2,\ldots,n\right\}
$. Let $G$ be the $m\times m$-matrix $\left(  \sum_{k=1}^{n}b_{i,k}%
d_{i,j,k}\right)  _{1\leq i\leq m,\ 1\leq j\leq m}$. Then,%
\[
\det G=\sum_{f:\left\{  1,2,\ldots,m\right\}  \rightarrow\left\{
1,2,\ldots,n\right\}  }\left(  \prod_{i=1}^{m}b_{i,f\left(  i\right)
}\right)  \det\left(  \left(  d_{i,j,f\left(  i\right)  }\right)  _{1\leq
i\leq m,\ 1\leq j\leq m}\right)  .
\]

\end{lemma}

Lemma \ref{lem.supergen.sum1} is merely a scary way to state the
multilinearity of the determinant as a function of its rows. See the Appendix
for a proof.

Let us specialize Lemma \ref{lem.supergen.sum1} in a way that is closer to our goal:

\begin{lemma}
\label{lem.supergen.sum2}Let $n$ be a positive integer. Let $b_{i,k}$ be an
element of $\mathbb{K}$ for every $i\in\left\{  1,2,\ldots,n-1\right\}  $ and
every $k\in\left\{  1,2,\ldots,n\right\}  $. Let $d_{i,j,k}$ be an element of
$\mathbb{K}$ for every $i\in\left\{  1,2,\ldots,n-1\right\}  $, $j\in\left\{
1,2,\ldots,n-1\right\}  $ and $k\in\left\{  1,2,\ldots,n\right\}  $. Let $G$
be the $\left(  n-1\right)  \times\left(  n-1\right)  $-matrix $\left(
\sum_{k=1}^{n}b_{i,k}d_{i,j,k}\right)  _{1\leq i\leq n-1,\ 1\leq j\leq n-1}$.
Then,%
\[
\det G=\sum_{\substack{f:\left\{  1,2,\ldots,n\right\}  \rightarrow\left\{
1,2,\ldots,n\right\}  ;\\f\left(  n\right)  =n}}\left(  \prod_{i=1}%
^{n-1}b_{i,f\left(  i\right)  }\right)  \det\left(  \left(  d_{i,j,f\left(
i\right)  }\right)  _{1\leq i\leq n-1,\ 1\leq j\leq n-1}\right)  .
\]

\end{lemma}

\begin{proof}
[Proof of Lemma \ref{lem.supergen.sum2}.]Lemma \ref{lem.supergen.sum1}
(applied to $m=n-1$) shows that%
\[
\det G=\sum_{f:\left\{  1,2,\ldots,n-1\right\}  \rightarrow\left\{
1,2,\ldots,n\right\}  }\left(  \prod_{i=1}^{n-1}b_{i,f\left(  i\right)
}\right)  \det\left(  \left(  d_{i,j,f\left(  i\right)  }\right)  _{1\leq
i\leq n-1,\ 1\leq j\leq n-1}\right)  .
\]
The only difference between this formula and the claim of Lemma
\ref{lem.supergen.sum2} is that the sum here is over all $f:\left\{
1,2,\ldots,n-1\right\}  \rightarrow\left\{  1,2,\ldots,n\right\}  $, whereas
the sum in the claim of Lemma \ref{lem.supergen.sum2} is over all $f:\left\{
1,2,\ldots,n\right\}  \rightarrow\left\{  1,2,\ldots,n\right\}  $ satisfying
$f\left(  n\right)  =n$. But this is not much of a difference: Each map
$\left\{  1,2,\ldots,n-1\right\}  \rightarrow\left\{  1,2,\ldots,n\right\}  $
is a restriction (to $\left\{  1,2,\ldots,n-1\right\}  $) of a unique map
$f:\left\{  1,2,\ldots,n\right\}  \rightarrow\left\{  1,2,\ldots,n\right\}  $
satisfying $f\left(  n\right)  =n$, and therefore the two sums are equal.
\end{proof}

\begin{proof}
[Proof of Theorem \ref{thm.supergen}.]For every $i\in\left\{  1,2,\ldots
,n-1\right\}  $, $j\in\left\{  1,2,\ldots,n-1\right\}  $ and $k\in\left\{
1,2,\ldots,n\right\}  $, define an element $d_{i,j,k}$ of $\mathbb{K}$ by%
\begin{equation}
d_{i,j,k}=a_{i,j}a_{k,n}-a_{i,n}a_{k,j}. \label{pf.thm.supergen.dijk}%
\end{equation}
For every $f:\left\{  1,2,\ldots,n\right\}  \rightarrow\left\{  1,2,\ldots
,n\right\}  $ satisfying $f\left(  n\right)  =n$, we have%
\begin{align}
&  \det\left(  \left(  \underbrace{d_{i,j,f\left(  i\right)  }}%
_{\substack{=a_{i,j}a_{f\left(  i\right)  ,n}-a_{i,n}a_{f\left(  i\right)
,j}\\\text{(by (\ref{pf.thm.supergen.dijk}))}}}\right)  _{1\leq i\leq
n-1,\ 1\leq j\leq n-1}\right) \nonumber\\
&  =\det\left(  \left(  a_{i,j}a_{f\left(  i\right)  ,n}-a_{i,n}a_{f\left(
i\right)  ,j}\right)  _{1\leq i\leq n-1,\ 1\leq j\leq n-1}\right) \nonumber\\
&  =%
\begin{cases}
0, & \text{if }f\text{ is not }n\text{-potent};\\
\left(  \operatorname*{abut}\nolimits_{f}A\right)  \cdot\det A, & \text{if
}f\text{ is }n\text{-potent}%
\end{cases}
\label{pf.thm.supergen.det1}%
\end{align}
(by Theorem \ref{thm.chio-gen}, applied to the matrix $\left(  a_{i,j}%
a_{f\left(  i\right)  ,n}-a_{i,n}a_{f\left(  i\right)  ,j}\right)  _{1\leq
i\leq n-1,\ 1\leq j\leq n-1}$ instead of $B$).

We have
\[
\left(  c_{i,j}\right)  _{1\leq i\leq n,\ 1\leq j\leq n}=BA=\left(  \sum
_{k=1}^{n}b_{i,k}a_{k,j}\right)  _{1\leq i\leq n,\ 1\leq j\leq n}%
\]
(by the definition of the product of two matrices). Thus,%
\begin{equation}
c_{i,j}=\sum_{k=1}^{n}b_{i,k}a_{k,j}\ \ \ \ \ \ \ \ \ \ \text{for every
}\left(  i,j\right)  \in\left\{  1,2,\ldots,n\right\}  ^{2}.
\label{pf.thm.supergen.cij}%
\end{equation}
Now, for every $\left(  i,j\right)  \in\left\{  1,2,\ldots,n-1\right\}  ^{2}$,
we have%
\begin{align*}
&  a_{i,j}\underbrace{c_{i,n}}_{\substack{=\sum_{k=1}^{n}b_{i,k}%
a_{k,n}\\\text{(by (\ref{pf.thm.supergen.cij}), applied to }n\\\text{instead
of }j\text{)}}}-a_{i,n}\underbrace{c_{i,j}}_{\substack{=\sum_{k=1}^{n}%
b_{i,k}a_{k,j}\\\text{(by (\ref{pf.thm.supergen.cij}))}}}\\
&  =a_{i,j}\sum_{k=1}^{n}b_{i,k}a_{k,n}-a_{i,n}\sum_{k=1}^{n}b_{i,k}%
a_{k,j}=\sum_{k=1}^{n}b_{i,k}\underbrace{\left(  a_{i,j}a_{k,n}-a_{i,n}%
a_{k,j}\right)  }_{\substack{=d_{i,j,k}\\\text{(by (\ref{pf.thm.supergen.dijk}%
))}}}=\sum_{k=1}^{n}b_{i,k}d_{i,j,k}.
\end{align*}
Hence,
\[
G=\left(  \underbrace{a_{i,j}c_{i,n}-a_{i,n}c_{i,j}}_{=\sum_{k=1}^{n}%
b_{i,k}d_{i,j,k}}\right)  _{1\leq i\leq n-1,\ 1\leq j\leq n-1}=\left(
\sum_{k=1}^{n}b_{i,k}d_{i,j,k}\right)  _{1\leq i\leq n-1,\ 1\leq j\leq n-1}.
\]
Hence, Lemma \ref{lem.supergen.sum2} yields%
\begin{align*}
\det G  &  =\sum_{\substack{f:\left\{  1,2,\ldots,n\right\}  \rightarrow
\left\{  1,2,\ldots,n\right\}  ;\\f\left(  n\right)  =n}}\underbrace{\left(
\prod_{i=1}^{n-1}b_{i,f\left(  i\right)  }\right)  }%
_{\substack{=\operatorname*{weight}\nolimits_{f}B\\\text{(by the
definition}\\\text{of }\operatorname*{weight}\nolimits_{f}B\text{)}%
}}\underbrace{\det\left(  \left(  d_{i,j,f\left(  i\right)  }\right)  _{1\leq
i\leq n-1,\ 1\leq j\leq n-1}\right)  }_{=%
\begin{cases}
0, & \text{if }f\text{ is not }n\text{-potent};\\
\left(  \operatorname*{abut}\nolimits_{f}A\right)  \cdot\det A, & \text{if
}f\text{ is }n\text{-potent}%
\end{cases}
}\\
&  =\sum_{\substack{f:\left\{  1,2,\ldots,n\right\}  \rightarrow\left\{
1,2,\ldots,n\right\}  ;\\f\left(  n\right)  =n}}\left(  \operatorname*{weight}%
\nolimits_{f}B\right)
\begin{cases}
0, & \text{if }f\text{ is not }n\text{-potent};\\
\left(  \operatorname*{abut}\nolimits_{f}A\right)  \cdot\det A, & \text{if
}f\text{ is }n\text{-potent}%
\end{cases}
\\
&  =\sum_{\substack{f:\left\{  1,2,\ldots,n\right\}  \rightarrow\left\{
1,2,\ldots,n\right\}  ;\\f\left(  n\right)  =n;\\f\text{ is }n\text{-potent}%
}}\left(  \operatorname*{weight}\nolimits_{f}B\right)  \left(
\operatorname*{abut}\nolimits_{f}A\right)  \cdot\det A\\
&  =\left(  \sum_{\substack{f:\left\{  1,2,\ldots,n\right\}  \rightarrow
\left\{  1,2,\ldots,n\right\}  ;\\f\left(  n\right)  =n;\\f\text{ is
}n\text{-potent}}}\left(  \operatorname*{weight}\nolimits_{f}B\right)  \left(
\operatorname*{abut}\nolimits_{f}A\right)  \right)  \cdot\det A.
\end{align*}

\end{proof}

\subsection{Deriving Theorem \ref{thm.mtt} from Theorem \ref{thm.supergen}}

Now let us see why Theorem \ref{thm.supergen} generalizes the matrix-tree theorem.

\begin{proof}
[Proof of Theorem \ref{thm.mtt}.]WLOG assume that $n\geq2$ (since the case
$n=1$ is easy to check by hand). Define an $n\times n$-matrix $A$ by
$A=\left(  a_{i,j}\right)  _{1\leq i\leq n,\ 1\leq j\leq n}$, where
\[
a_{i,j}=\delta_{i,j}+\delta_{j,n}\left(  1-\delta_{i,n}\right)  .
\]
(This scary formula hides a simple idea: this is the matrix whose entries on
the diagonal and in its last column are $1$, and all other entries are $0$.
Thus,%
\[
A=\left(
\begin{array}
[c]{ccccccc}%
1 & 0 & 0 & 0 & \cdots & 0 & 1\\
0 & 1 & 0 & 0 & \cdots & 0 & 1\\
0 & 0 & 1 & 0 & \cdots & 0 & 1\\
0 & 0 & 0 & 1 & \cdots & 0 & 1\\
\vdots & \vdots & \vdots & \vdots & \ddots & \vdots & \vdots\\
0 & 0 & 0 & 0 & \cdots & 1 & 1\\
0 & 0 & 0 & 0 & \cdots & 0 & 1
\end{array}
\right)  .
\]
) Note that every $\left(  i,j\right)  \in\left\{  1,2,\ldots,n-1\right\}
^{2}$ satisfies%
\begin{equation}
a_{i,j}=\delta_{i,j}+\underbrace{\delta_{j,n}}_{\substack{=0\\\text{(since
}j\neq n\\\text{(since }j\in\left\{  1,2,\ldots,n-1\right\}  \text{))}%
}}\left(  1-\delta_{i,n}\right)  =\delta_{i,j}. \label{pf.thm.mtt.aij=}%
\end{equation}
Also, every $i\in\left\{  1,2,\ldots,n-1\right\}  $ satisfies%
\begin{align}
a_{i,n}  &  =\underbrace{\delta_{i,n}}_{\substack{=0\\\text{(since }i\neq
n\text{)}}}+\underbrace{\delta_{n,n}}_{\substack{=1\\\text{(since }%
n=n\text{)}}}\left(  1-\underbrace{\delta_{i,n}}_{\substack{=0\\\text{(since
}i\neq n\text{)}}}\right)  \ \ \ \ \ \ \ \ \ \ \left(  \text{by the definition
of }a_{i,n}\right) \nonumber\\
&  =0+1\left(  1-0\right)  =1. \label{pf.thm.mtt.ain=}%
\end{align}

Also, let $B$ be the $n\times n$-matrix $\left(  W\left(  i,j\right)  \right)
_{1\leq i\leq n,\ 1\leq j\leq n}$. Write the $n\times n$-matrix $BA$ in the
form $BA=\left(  c_{i,j}\right)  _{1\leq i\leq n,\ 1\leq j\leq n}$. Then, it
is easy to see that every $\left(  i,j\right)  \in\left\{  1,2,\ldots
,n\right\}  ^{2}$ satisfies%
\begin{equation}
c_{i,j}=W\left(  i,j\right)  +\delta_{j,n}\left(  d^{+}\left(  i\right)
-W\left(  i,n\right)  \right)  \label{pf.thm.mtt.cij=}%
\end{equation}
\footnote{\textit{Proof of (\ref{pf.thm.mtt.cij=}):} For every $i\in\left\{
1,2,\ldots,n\right\}  $, we have%
\begin{align*}
d^{+}\left(  i\right)   &  =\sum_{j=1}^{n}W\left(  i,j\right)
\ \ \ \ \ \ \ \ \ \ \left(  \text{by the definition of }d^{+}\left(  i\right)
\right) \\
&  =\sum_{j=1}^{n-1}W\left(  i,j\right)  +W\left(  i,n\right)  =\sum
_{k=1}^{n-1}W\left(  i,k\right)  +W\left(  i,n\right)
\end{align*}
(here, we renamed the summation index $j$ as $k$) and thus%
\begin{equation}
\sum_{k=1}^{n-1}W\left(  i,k\right)  =d^{+}\left(  i\right)  -W\left(
i,n\right)  . \label{pf.thm.mtt.fn1.1}%
\end{equation}
\par
But%
\[
\left(  c_{i,j}\right)  _{1\leq i\leq n,\ 1\leq j\leq n}=BA=\left(  \sum
_{k=1}^{n}W\left(  i,k\right)  a_{k,j}\right)  _{1\leq i\leq n,\ 1\leq j\leq
n}%
\]
(by the definition of the product of two matrices, since $B=\left(  W\left(
i,j\right)  \right)  _{1\leq i\leq n,\ 1\leq j\leq n}$ and $A=\left(
a_{i,j}\right)  _{1\leq i\leq n,\ 1\leq j\leq n}$). Hence, every $\left(
i,j\right)  \in\left\{  1,2,\ldots,n\right\}  ^{2}$ satisfies%
\begin{align*}
c_{i,j}  &  =\sum_{k=1}^{n}W\left(  i,k\right)  \underbrace{a_{k,j}%
}_{\substack{=\delta_{k,j}+\delta_{j,n}\left(  1-\delta_{k,n}\right)
\\\text{(by the definition of }a_{k,j}\text{)}}}\\
&  =\sum_{k=1}^{n}W\left(  i,k\right)  \left(  \delta_{k,j}+\delta
_{j,n}\left(  1-\delta_{k,n}\right)  \right) \\
&  =\underbrace{\sum_{k=1}^{n}W\left(  i,k\right)  \delta_{k,j}}%
_{\substack{=W\left(  i,j\right)  \\\text{(because the factor }\delta
_{k,j}\text{ in the sum}\\\text{kills every addend except the one for
}k=j\text{)}}}+\delta_{j,n}\underbrace{\sum_{k=1}^{n}W\left(  i,k\right)
\left(  1-\delta_{k,n}\right)  }_{=\sum_{k=1}^{n-1}W\left(  i,k\right)
\left(  1-\delta_{k,n}\right)  +W\left(  i,n\right)  \left(  1-\delta
_{n,n}\right)  }\\
&  =W\left(  i,j\right)  +\delta_{j,n}\left(  \sum_{k=1}^{n-1}W\left(
i,k\right)  \left(  1-\underbrace{\delta_{k,n}}_{\substack{=0\\\text{(since
}k<n\text{)}}}\right)  +W\left(  i,n\right)  \underbrace{\left(
1-\delta_{n,n}\right)  }_{\substack{=0\\\text{(since }\delta_{n,n}=1\text{)}%
}}\right) \\
&  =W\left(  i,j\right)  +\delta_{j,n}\left(  \sum_{k=1}^{n-1}W\left(
i,k\right)  \underbrace{\left(  1-0\right)  }_{=1}+\underbrace{W\left(
i,n\right)  0}_{=0}\right) \\
&  =W\left(  i,j\right)  +\delta_{j,n}\underbrace{\sum_{k=1}^{n-1}W\left(
i,k\right)  }_{\substack{=d^{+}\left(  i\right)  -W\left(  i,n\right)
\\\text{(by (\ref{pf.thm.mtt.fn1.1}))}}}=W\left(  i,j\right)  +\delta
_{j,n}\left(  d^{+}\left(  i\right)  -W\left(  i,n\right)  \right)  ,
\end{align*}
and thus (\ref{pf.thm.mtt.cij=}) is proven.}.

Thus, for every $\left(  i,j\right)  \in\left\{  1,2,\ldots,n-1\right\}  ^{2}%
$, we have
\begin{align*}
&  \underbrace{a_{i,j}}_{\substack{=\delta_{i,j}\\\text{(by
(\ref{pf.thm.mtt.aij=}))}}}\underbrace{c_{i,n}}_{\substack{=W\left(
i,n\right)  +\delta_{n,n}\left(  d^{+}\left(  i\right)  -W\left(  i,n\right)
\right)  \\\text{(by (\ref{pf.thm.mtt.cij=}), applied}\\\text{to }n\text{
instead of }j\text{)}}}-\underbrace{a_{i,n}}_{\substack{=1\\\text{(by
(\ref{pf.thm.mtt.ain=}))}}}\underbrace{c_{i,j}}_{\substack{=W\left(
i,j\right)  +\delta_{j,n}\left(  d^{+}\left(  i\right)  -W\left(  i,n\right)
\right)  \\\text{(by (\ref{pf.thm.mtt.cij=}))}}}\\
&  =\delta_{i,j}\left(  W\left(  i,n\right)  +\underbrace{\delta_{n,n}}%
_{=1}\left(  d^{+}\left(  i\right)  -W\left(  i,n\right)  \right)  \right)
-\left(  W\left(  i,j\right)  +\underbrace{\delta_{j,n}}%
_{\substack{=0\\\text{(since }j<n\text{)}}}\left(  d^{+}\left(  i\right)
-W\left(  i,n\right)  \right)  \right)  \\
&  =\delta_{i,j}\underbrace{\left(  W\left(  i,n\right)  +\left(  d^{+}\left(
i\right)  -W\left(  i,n\right)  \right)  \right)  }_{=d^{+}\left(  i\right)
}-W\left(  i,j\right)  =\delta_{i,j}d^{+}\left(  i\right)  -W\left(
i,j\right)  .
\end{align*}
Hence,%
\[
\left(  a_{i,j}c_{i,n}-a_{i,n}c_{i,j}\right)  _{1\leq i\leq n-1,\ 1\leq j\leq
n-1}=\left(  \delta_{i,j}d^{+}\left(  i\right)  -W\left(  i,j\right)  \right)
_{1\leq i\leq n-1,\ 1\leq j\leq n-1}=L.
\]
In other words, $L$ is the matrix $\left(  a_{i,j}c_{i,n}-a_{i,n}%
c_{i,j}\right)  _{1\leq i\leq n-1,\ 1\leq j\leq n-1}\in\mathbb{K}^{\left(
n-1\right)  \times\left(  n-1\right)  }$. Thus, Theorem \ref{thm.supergen}
(applied to $G=L$) yields%
\begin{align*}
\det L &  =\left(  \sum_{\substack{f:\left\{  1,2,\ldots,n\right\}
\rightarrow\left\{  1,2,\ldots,n\right\}  ;\\f\left(  n\right)  =n;\\f\text{
is }n\text{-potent}}}\underbrace{\left(  \operatorname*{weight}\nolimits_{f}%
B\right)  }_{=\prod_{i=1}^{n-1}W\left(  i,f\left(  i\right)  \right)
}\underbrace{\left(  \operatorname*{abut}\nolimits_{f}A\right)  }_{=1}\right)
\cdot\underbrace{\det A}_{=1}\\
&  =\sum_{\substack{f:\left\{  1,2,\ldots,n\right\}  \rightarrow\left\{
1,2,\ldots,n\right\}  ;\\f\left(  n\right)  =n;\\f\text{ is }n\text{-potent}%
}}\prod_{i=1}^{n-1}W\left(  i,f\left(  i\right)  \right)  .
\end{align*}
This proves Theorem \ref{thm.mtt}.
\end{proof}

\subsection{Some combinatorial lemmas}

We still owe the reader a proof of Theorem \ref{thm.chio-gen}. We prepare by
proving some properties of maps $f:\left\{  1,2,\ldots,n\right\}
\rightarrow\left\{  1,2,\ldots,n\right\}  $.

\begin{proposition}
\label{prop.map.image}Let $n\in\mathbb{N}$. Let $f:\left\{  1,2,\ldots
,n\right\}  \rightarrow\left\{  1,2,\ldots,n\right\}  $ be a map. Let
$i\in\left\{  1,2,\ldots,n\right\}  $. Then,%
\[
f^{k}\left(  i\right)  \in\left\{  f^{s}\left(  i\right)  \ \mid\ s\in\left\{
0,1,\ldots,n-1\right\}  \right\}  \ \ \ \ \ \ \ \ \ \ \text{for every }%
k\in\mathbb{N}.
\]

\end{proposition}

Proposition \ref{prop.map.image} is a classical fact; we give the proof in the
Appendix below.

The following three results can be easily derived from Proposition
\ref{prop.map.image}; we shall give more detailed proofs in the Appendix:

\begin{proposition}
\label{prop.map.preim}Let $n$ be a positive integer. Let $f:\left\{
1,2,\ldots,n\right\}  \rightarrow\left\{  1,2,\ldots,n\right\}  $ be a map
such that $f\left(  n\right)  =n$. Let $i\in\left\{  1,2,\ldots,n\right\}  $.
Then, $f^{n-1}\left(  i\right)  =n$ if and only if there exists some
$k\in\mathbb{N}$ such that $f^{k}\left(  i\right)  =n$.
\end{proposition}

\begin{proposition}
\label{prop.n-potent.calc}Let $n$ be a positive integer. Let $f:\left\{
1,2,\ldots,n\right\}  \rightarrow\left\{  1,2,\ldots,n\right\}  $ be a map
such that $f\left(  n\right)  =n$. Then, the map $f$ is $n$-potent if and only
if $f^{n-1}\left(  \left\{  1,2,\ldots,n\right\}  \right)  =\left\{
n\right\}  $.
\end{proposition}

\begin{corollary}
\label{cor.n-potent.delta}Let $n$ be a positive integer. Let $f:\left\{
1,2,\ldots,n\right\}  \rightarrow\left\{  1,2,\ldots,n\right\}  $ be a map
such that $f\left(  n\right)  =n$. Let $i\in\left\{  1,2,\ldots,n\right\}  $.
Then, $\delta_{f^{n-1}\left(  i\right)  ,n}=\delta_{f^{n}\left(  i\right)
,n}$.
\end{corollary}

One consequence of Proposition \ref{prop.n-potent.calc} is the following: If
$n$ is a positive integer, and if $f:\left\{  1,2,\ldots,n\right\}
\rightarrow\left\{  1,2,\ldots,n\right\}  $ is a map such that $f\left(
n\right)  =n$, then we can check in finite time whether the map $f$ is
$n$-potent (because we can check in finite time whether $f^{n-1}\left(
\left\{  1,2,\ldots,n\right\}  \right)  =\left\{  n\right\}  $). Thus, for any
given positive integer $n$, it is possible to enumerate all $n$-potent maps
$f:\left\{  1,2,\ldots,n\right\}  \rightarrow\left\{  1,2,\ldots,n\right\}  $.

Next, we shall show a property of $n$-potent maps:

\begin{lemma}
\label{lem.n-potent.sigma}Let $n$ be a positive integer. Let $f:\left\{
1,2,\ldots,n\right\}  \rightarrow\left\{  1,2,\ldots,n\right\}  $ be a map
such that $f\left(  n\right)  =n$. Assume that $f$ is $n$-potent.

Let $\sigma\in S_{n}$ be a permutation such that $\sigma\neq\operatorname*{id}%
$. Then, there exists some $i\in\left\{  1,2,\ldots,n\right\}  $ such that
$\sigma\left(  i\right)  \notin\left\{  i,f\left(  i\right)  \right\}  $.
\end{lemma}

\begin{proof}
[Proof of Lemma \ref{lem.n-potent.sigma}.]Assume the contrary. Thus,
$\sigma\left(  i\right)  \in\left\{  i,f\left(  i\right)  \right\}  $ for
every $i\in\left\{  1,2,\ldots,n\right\}  $.

We have $\sigma\neq\operatorname*{id}$. Hence, there exists some $h\in\left\{
1,2,\ldots,n\right\}  $ such that $\sigma\left(  h\right)  \neq h$. Fix such a
$h$. We shall prove that%
\begin{equation}
\sigma^{j}\left(  h\right)  =f^{j}\left(  h\right)
\ \ \ \ \ \ \ \ \ \ \text{for every }j\in\mathbb{N}.
\label{pf.thm.chio-gen.alg.fn4.1}%
\end{equation}

Indeed, we shall prove this by induction over $j$. The induction base (the
case $j=0$) is obvious. For the induction step, fix $J\in\mathbb{N}$, and
assume that $\sigma^{J}\left(  h\right)  =f^{J}\left(  h\right)  $. We need to
prove that $\sigma^{J+1}\left(  h\right)  =f^{J+1}\left(  h\right)  $.

We have assumed that $\sigma\left(  i\right)  \in\left\{  i,f\left(  i\right)
\right\}  $ for every $i\in\left\{  1,2,\ldots,n\right\}  $. Applying this to
$i=\sigma^{J}\left(  h\right)  $, we obtain $\sigma\left(  \sigma^{J}\left(
h\right)  \right)  \in\left\{  \sigma^{J}\left(  h\right)  ,f\left(
\sigma^{J}\left(  h\right)  \right)  \right\}  $. In other words,
$\sigma^{J+1}\left(  h\right)  \in\left\{  \sigma^{J}\left(  h\right)
,f\left(  \sigma^{J}\left(  h\right)  \right)  \right\}  $. Thus, either
$\sigma^{J+1}\left(  h\right)  =\sigma^{J}\left(  h\right)  $ or $\sigma
^{J+1}\left(  h\right)  =f\left(  \sigma^{J}\left(  h\right)  \right)  $.
Since $\sigma^{J+1}\left(  h\right)  =\sigma^{J}\left(  h\right)  $ is
impossible (because in light of the invertibility of $\sigma$, this would
yield $\sigma\left(  h\right)  =h$, which contradicts $\sigma\left(  h\right)
\neq h$), we thus must have $\sigma^{J+1}\left(  h\right)  =f\left(
\sigma^{J}\left(  h\right)  \right)  $. Hence, $\sigma^{J+1}\left(  h\right)
=f\left(  \underbrace{\sigma^{J}\left(  h\right)  }_{=f^{J}\left(  h\right)
}\right)  =f\left(  f^{J}\left(  h\right)  \right)  =f^{J+1}\left(  h\right)
$. This completes the induction step.

Thus, (\ref{pf.thm.chio-gen.alg.fn4.1}) is proven.

But $f$ is $n$-potent. Hence, there exists some $k\in\mathbb{N}$ such that
$f^{k}\left(  h\right)  =n$. Consider this $k$. Applying
(\ref{pf.thm.chio-gen.alg.fn4.1}) to $j=k$, we obtain $\sigma^{k}\left(
h\right)  =f^{k}\left(  h\right)  =n$.

But applying (\ref{pf.thm.chio-gen.alg.fn4.1}) to $j=k+1$, we obtain
$\sigma^{k+1}\left(  h\right)  =f^{k+1}\left(  h\right)  =f\left(
\underbrace{f^{k}\left(  h\right)  }_{=n}\right)  =f\left(  n\right)  =n$.
Hence, $n=\sigma^{k+1}\left(  h\right)  =\sigma^{k}\left(  \sigma\left(
h\right)  \right)  $, so that $\sigma^{k}\left(  \sigma\left(  h\right)
\right)  =n=\sigma^{k}\left(  h\right)  $. Since $\sigma^{k}$ is invertible,
this entails $\sigma\left(  h\right)  =h$, which contradicts $\sigma\left(
h\right)  \neq h$. This contradiction proves that our assumption was wrong.
Thus, Lemma \ref{lem.n-potent.sigma} is proven.
\end{proof}

\subsection{The matrix $Z_{f}$ and its determinant}

Next, we assign a matrix $Z_{f}$ to every such $f:\left\{  1,2,\ldots
,n\right\}  \rightarrow\left\{  1,2,\ldots,n\right\}  $:

\begin{definition}
\label{def.Zf}Let $n$ be a positive integer. Let $f:\left\{  1,2,\ldots
,n\right\}  \rightarrow\left\{  1,2,\ldots,n\right\}  $ be a map. Then, we
define an $n\times n$-matrix $Z_{f}\in\mathbb{K}^{n\times n}$ by%
\[
Z_{f}=\left(  \delta_{i,j}-\left(  1-\delta_{i,n}\right)  \delta_{f\left(
i\right)  ,j}\right)  _{1\leq i\leq n,\ 1\leq j\leq n}.
\]

\end{definition}

\begin{example}
For this example, set $n=4$, and define a map $f:\left\{  1,2,3,4\right\}
\rightarrow\left\{  1,2,3,4\right\}  $ by $\left(  f\left(  1\right)
,f\left(  2\right)  ,f\left(  3\right)  ,f\left(  4\right)  \right)  =\left(
2,4,1,4\right)  $. Then,%
\[
Z_{f}=\left(
\begin{array}
[c]{cccc}%
1 & -1 & 0 & 0\\
0 & 1 & 0 & -1\\
-1 & 0 & 1 & 0\\
0 & 0 & 0 & 1
\end{array}
\right)  .
\]

\end{example}

Now, we claim the following:

\begin{proposition}
\label{prop.n-potent.Zv}Let $n$ be a positive integer. Let $f:\left\{
1,2,\ldots,n\right\}  \rightarrow\left\{  1,2,\ldots,n\right\}  $ be a map
such that $f\left(  n\right)  =n$. Let $v_{f}$ be the column vector $\left(
1-\delta_{f^{n-1}\left(  i\right)  ,n}\right)  _{1\leq i\leq n,\ 1\leq j\leq
1}\in\mathbb{K}^{n\times1}$. Then, $Z_{f}v_{f}=0_{n\times1}$.

(Recall that $0_{n\times1}$ denotes the $n\times1$ zero matrix, i.e., the
column vector with $n$ entries whose all entries are $0$.)
\end{proposition}

\begin{proof}
[Proof of Proposition \ref{prop.n-potent.Zv}.]We shall prove that
\begin{equation}
\sum_{k=1}^{n}\left(  \delta_{i,k}-\left(  1-\delta_{i,n}\right)
\delta_{f\left(  i\right)  ,k}\right)  \left(  1-\delta_{f^{n-1}\left(
k\right)  ,n}\right)  =0 \label{pf.prop.n-potent.Zv.1}%
\end{equation}
for every $i\in\left\{  1,2,\ldots,n\right\}  $.

\textit{Proof of (\ref{pf.prop.n-potent.Zv.1}):} Let $i\in\left\{
1,2,\ldots,n\right\}  $. Corollary \ref{cor.n-potent.delta} yields
$\delta_{f^{n-1}\left(  i\right)  ,n}=\delta_{f^{n}\left(  i\right)  ,n}$.

On the other hand, $f\left(  n\right)  =n$. Thus, it is straightforward to see
(by induction over $h$) that $f^{h}\left(  n\right)  =n$ for every
$h\in\mathbb{N}$. Applying this to $h=n$, we obtain $f^{n}\left(  n\right)
=n$.

Now,%
\begin{align*}
&  \sum_{k=1}^{n}\left(  \delta_{i,k}-\left(  1-\delta_{i,n}\right)
\delta_{f\left(  i\right)  ,k}\right)  \left(  1-\delta_{f^{n-1}\left(
k\right)  ,n}\right) \\
&  =\underbrace{\sum_{k=1}^{n}\delta_{i,k}\left(  1-\delta_{f^{n-1}\left(
k\right)  ,n}\right)  }_{\substack{=1-\delta_{f^{n-1}\left(  i\right)
,n}\\\text{(because the factor }\delta_{i,k}\text{ in the sum}\\\text{kills
every addend except the one for }k=i\text{)}}}-\underbrace{\sum_{k=1}%
^{n}\left(  1-\delta_{i,n}\right)  \delta_{f\left(  i\right)  ,k}\left(
1-\delta_{f^{n-1}\left(  k\right)  ,n}\right)  }_{\substack{=\left(
1-\delta_{i,n}\right)  \left(  1-\delta_{f^{n-1}\left(  f\left(  i\right)
\right)  ,n}\right)  \\\text{(because the factor }\delta_{f\left(  i\right)
,k}\text{ in the sum}\\\text{kills every addend except the one for }k=f\left(
i\right)  \text{)}}}\\
&  =\left(  1-\underbrace{\delta_{f^{n-1}\left(  i\right)  ,n}}_{=\delta
_{f^{n}\left(  i\right)  ,n}}\right)  -\left(  1-\delta_{i,n}\right)  \left(
1-\underbrace{\delta_{f^{n-1}\left(  f\left(  i\right)  \right)  ,n}}%
_{=\delta_{f^{n}\left(  i\right)  ,n}}\right) \\
&  =\left(  1-\delta_{f^{n}\left(  i\right)  ,n}\right)  -\left(
1-\delta_{i,n}\right)  \left(  1-\delta_{f^{n}\left(  i\right)  ,n}\right) \\
&  =\underbrace{\left(  1-\left(  1-\delta_{i,n}\right)  \right)  }%
_{=\delta_{i,n}}\left(  1-\delta_{f^{n}\left(  i\right)  ,n}\right)
=\delta_{i,n}\left(  1-\delta_{f^{n}\left(  i\right)  ,n}\right) \\
&  =%
\begin{cases}
0, & \text{if }i\neq n;\\
1-\delta_{f^{n}\left(  n\right)  ,n}, & \text{if }i=n
\end{cases}
=%
\begin{cases}
0, & \text{if }i\neq n;\\
0, & \text{if }i=n
\end{cases}
\\
&  \ \ \ \ \ \ \ \ \ \ \left(  \text{since }f^{n}\left(  n\right)  =n\text{
and thus }\delta_{f^{n}\left(  n\right)  ,n}=\delta_{n,n}=1\text{ and hence
}1-\delta_{f^{n}\left(  n\right)  ,n}=0\right) \\
&  =0.
\end{align*}
This proves (\ref{pf.prop.n-potent.Zv.1}).

Recall now that
\[
Z_{f}=\left(  \delta_{i,j}-\left(  1-\delta_{i,n}\right)  \delta_{f\left(
i\right)  ,j}\right)  _{1\leq i\leq n,\ 1\leq j\leq n}%
\]
and $v_{f}=\left(  1-\delta_{f^{n-1}\left(  i\right)  ,n}\right)  _{1\leq
i\leq n,\ 1\leq j\leq1}$. Hence, the definition of the product of two matrices
yields%
\begin{align*}
Z_{f}v_{f}  &  =\left(  \underbrace{\sum_{k=1}^{n}\left(  \delta_{i,k}-\left(
1-\delta_{i,n}\right)  \delta_{f\left(  i\right)  ,k}\right)  \left(
1-\delta_{f^{n-1}\left(  k\right)  ,n}\right)  }_{\substack{=0\\\text{(by
(\ref{pf.prop.n-potent.Zv.1}))}}}\right)  _{1\leq i\leq n,\ 1\leq j\leq1}\\
&  =\left(  0\right)  _{1\leq i\leq n,\ 1\leq j\leq1}=0_{n\times1}.
\end{align*}
This proves Proposition \ref{prop.n-potent.Zv}.
\end{proof}

Now, we recall the following well-known properties of
determinants\footnote{For the sake of completeness: Lemma \ref{lem.adj.kernel}
is \cite[Corollary 6.102]{detnotes}; Lemma \ref{lem.det.ain=0} is
\cite[Corollary 6.45]{detnotes}.}:

\begin{lemma}
\label{lem.adj.kernel}Let $n\in\mathbb{N}$. Let $A$ be an $n\times n$-matrix.
Let $v$ be a column vector with $n$ entries. If $Av=0_{n\times1}$, then $\det
A\cdot v=0_{n\times1}$.
\end{lemma}

\begin{lemma}
\label{lem.det.ain=0}Let $n$ be a positive integer. Let $A=\left(
a_{i,j}\right)  _{1\leq i\leq n,\ 1\leq j\leq n}$ be an $n\times n$-matrix.
Assume that%
\begin{equation}
a_{i,n}=0\ \ \ \ \ \ \ \ \ \ \text{for every }i\in\left\{  1,2,\ldots
,n-1\right\}  . \label{eq.cor.laplace.pre.col.ass}%
\end{equation}
Then, $\det A=a_{n,n}\cdot\det\left(  \left(  a_{i,j}\right)  _{1\leq i\leq
n-1,\ 1\leq j\leq n-1}\right)  $.
\end{lemma}

Now, we can prove the crucial property of the matrix $Z_{f}$:

\begin{proposition}
\label{prop.Zf.n-pot}Let $n$ be a positive integer. Let $f:\left\{
1,2,\ldots,n\right\}  \rightarrow\left\{  1,2,\ldots,n\right\}  $ be a map
satisfying $f\left(  n\right)  =n$.

\textbf{(a)} If $f$ is $n$-potent, then $\det\left(  Z_{f}\right)  =1$.

\textbf{(b)} If $f$ is not $n$-potent, then $\det\left(  Z_{f}\right)  =0$.
\end{proposition}

\begin{proof}
[Proof of Proposition \ref{prop.Zf.n-pot}.]Write the matrix $Z_{f}$ in the
form $\left(  z_{i,j}\right)  _{1\leq i\leq n,\ 1\leq j\leq n}$. Thus,%
\[
\left(  z_{i,j}\right)  _{1\leq i\leq n,\ 1\leq j\leq n}=Z_{f}=\left(
\delta_{i,j}-\left(  1-\delta_{i,n}\right)  \delta_{f\left(  i\right)
,j}\right)  _{1\leq i\leq n,\ 1\leq j\leq n}.
\]
Hence, every $\left(  i,j\right)  \in\left\{  1,2,\ldots,n\right\}  ^{2}$
satisfies%
\begin{align}
z_{i,j}  &  =\delta_{i,j}-\underbrace{\left(  1-\delta_{i,n}\right)  }_{=%
\begin{cases}
1, & \text{if }i<n;\\
0, & \text{if }i=n
\end{cases}
}\delta_{f\left(  i\right)  ,j}=\delta_{i,j}-%
\begin{cases}
1, & \text{if }i<n;\\
0, & \text{if }i=n
\end{cases}
\delta_{f\left(  i\right)  ,j}\label{pf.prop.Zf.n-pot.zij==}\\
&  =\delta_{i,j}-%
\begin{cases}
\delta_{f\left(  i\right)  ,j}, & \text{if }i<n;\\
0, & \text{if }i=n
\end{cases}
=%
\begin{cases}
\delta_{i,j}-\delta_{f\left(  i\right)  ,j}, & \text{if }i<n;\\
\delta_{i,j}, & \text{if }i=n
\end{cases}
. \label{pf.prop.Zf.n-pot.zij=}%
\end{align}

\textbf{(a)} Assume that $f$ is $n$-potent.

Let $\sigma\in S_{n}$ be a permutation such that $\sigma\neq\operatorname*{id}%
$. Then, there exists some $i\in\left\{  1,2,\ldots,n\right\}  $ such that
$\sigma\left(  i\right)  \notin\left\{  i,f\left(  i\right)  \right\}  $ (by
Lemma \ref{lem.n-potent.sigma}). Hence, there exists some $i\in\left\{
1,2,\ldots,n\right\}  $ such that $z_{i,\sigma\left(  i\right)  }%
=0$\ \ \ \ \footnote{\textit{Proof.} We have just shown that there exists some
$i\in\left\{  1,2,\ldots,n\right\}  $ such that $\sigma\left(  i\right)
\notin\left\{  i,f\left(  i\right)  \right\}  $. Consider this $i$. We have
$\sigma\left(  i\right)  \notin\left\{  i,f\left(  i\right)  \right\}  $, thus
$\sigma\left(  i\right)  \neq i$, and thus $\delta_{i,\sigma\left(  i\right)
}=0$. Also, $\sigma\left(  i\right)  \notin\left\{  i,f\left(  i\right)
\right\}  $, thus $\sigma\left(  i\right)  \neq f\left(  i\right)  $, and thus
$\delta_{f\left(  i\right)  ,\sigma\left(  i\right)  }=0$. Now,
(\ref{pf.prop.Zf.n-pot.zij==}) (applied to $\left(  i,\sigma\left(  i\right)
\right)  $ instead of $\left(  i,j\right)  $) yields%
\[
z_{i,\sigma\left(  i\right)  }=\underbrace{\delta_{i,\sigma\left(  i\right)
}}_{=0}-%
\begin{cases}
1, & \text{if }i<n;\\
0, & \text{if }i=n
\end{cases}
\underbrace{\delta_{f\left(  i\right)  ,\sigma\left(  i\right)  }}%
_{=0}=0-0=0,
\]
qed.}. Hence, the product $\prod\limits_{i=1}^{n}z_{i,\sigma\left(  i\right)
}$ has at least one zero factor, and thus equals $0$.

Now, forget that we fixed $\sigma$. We thus have shown that
\begin{equation}
\prod\limits_{i=1}^{n}z_{i,\sigma\left(  i\right)  }%
=0\ \ \ \ \ \ \ \ \ \ \text{for every }\sigma\in S_{n}\text{ such that }%
\sigma\neq\operatorname*{id}. \label{pf.thm.chio-gen.alg.b.almost}%
\end{equation}
On the other hand, it is easy to see that%
\begin{equation}
\prod_{i=1}^{n}z_{i,i}=1. \label{pf.thm.chio-gen.alg.b.almost2}%
\end{equation}
\footnote{\textit{Proof of (\ref{pf.thm.chio-gen.alg.b.almost2}):} To prove
this, it is sufficient to show that $z_{i,i}=1$ for every $i\in\left\{
1,2,\ldots,n\right\}  $. This is obvious when $i=n$ (using the formula
(\ref{pf.prop.Zf.n-pot.zij=})), so we only need to consider the case when
$i<n$. In this case, (\ref{pf.prop.Zf.n-pot.zij=}) (applied to $\left(
i,i\right)  $ instead of $\left(  i,j\right)  $) shows that $z_{i,i}%
=\underbrace{\delta_{i,i}}_{=1}-\delta_{f\left(  i\right)  ,i}=1-\delta
_{f\left(  i\right)  ,i}$. Hence, in order to prove that $z_{i,i}=1$, we need
to show that $\delta_{f\left(  i\right)  ,i}=0$. In other words, we need to
prove that $f\left(  i\right)  \neq i$.
\par
Indeed, assume the contrary. Thus, $f\left(  i\right)  =i$. Hence, by
induction over $k$, we can easily see that $f^{k}\left(  i\right)  =i$ for
every $k\in\mathbb{N}$. Hence, for every $k\in\mathbb{N}$, we have
$f^{k}\left(  i\right)  =i\neq n$. This contradicts the fact that there exists
some $k\in\mathbb{N}$ such that $f^{k}\left(  i\right)  =n$ (since $f$ is
$n$-potent). This contradiction proves that our assumption was wrong. Hence,
(\ref{pf.thm.chio-gen.alg.b.almost2}) is proven.}

Now, the definition of $\det\left(  Z_{f}\right)  $ yields%
\begin{align*}
\det\left(  Z_{f}\right)   &  =\sum_{\sigma\in S_{n}}\left(  -1\right)
^{\sigma}\prod_{i=1}^{n}z_{i,\sigma\left(  i\right)  }%
\ \ \ \ \ \ \ \ \ \ \left(  \text{since }Z_{f}=\left(  z_{i,j}\right)  _{1\leq
i\leq n,\ 1\leq j\leq n}\right) \\
&  =\underbrace{\left(  -1\right)  ^{\operatorname*{id}}}_{=1}\prod_{i=1}%
^{n}\underbrace{z_{i,\operatorname*{id}\left(  i\right)  }}_{=z_{i,i}}%
+\sum_{\substack{\sigma\in S_{n};\\\sigma\neq\operatorname*{id}}}\left(
-1\right)  ^{\sigma}\underbrace{\prod\limits_{i=1}^{n}z_{i,\sigma\left(
i\right)  }}_{\substack{=0\\\text{(by (\ref{pf.thm.chio-gen.alg.b.almost}))}%
}}\\
&  =\prod_{i=1}^{n}z_{i,i}+\underbrace{\sum_{\substack{\sigma\in
S_{n};\\\sigma\neq\operatorname*{id}}}\left(  -1\right)  ^{\sigma}0}%
_{=0}=\prod_{i=1}^{n}z_{i,i}=1\ \ \ \ \ \ \ \ \ \ \left(  \text{by
(\ref{pf.thm.chio-gen.alg.b.almost2})}\right)  .
\end{align*}
This proves Proposition \ref{prop.Zf.n-pot} \textbf{(a)}.

\textbf{(b)} Assume that $f$ is not $n$-potent. Then, there exists some
$i\in\left\{  1,2,\ldots,n\right\}  $ such that $f^{n-1}\left(  i\right)  \neq
n$\ \ \ \ \footnote{\textit{Proof.} Assume the contrary. Thus, for every
$i\in\left\{  1,2,\ldots,n\right\}  $, we have $f^{n-1}\left(  i\right)  =n$.
Hence, for every $i\in\left\{  1,2,\ldots,n\right\}  $, there exists some
$k\in\mathbb{N}$ such that $f^{k}\left(  i\right)  =n$ (namely, $k=n-1$). In
other words, the map $f$ is $n$-potent. This contradicts the fact that $f$ is
not $n$-potent. This contradiction shows that our assumption was wrong, qed.}.
Fix such an $i$, and denote it by $u$. Thus, $u\in\left\{  1,2,\ldots
,n\right\}  $ is such that $f^{n-1}\left(  u\right)  \neq n$.

Define the vector $v_{f}$ as in Proposition \ref{prop.n-potent.Zv}.
Proposition \ref{prop.n-potent.Zv} yields $Z_{f}v_{f}=0_{n\times1}$. Lemma
\ref{lem.adj.kernel} (applied to $Z_{f}$ and $v_{f}$ instead of $A$ and $v$)
thus yields $\det\left(  Z_{f}\right)  \cdot v_{f}=0_{n\times1}$. Thus,%
\begin{align*}
\left(  0\right)  _{1\leq i\leq n,\ 1\leq j\leq1}  &  =0_{n\times1}%
=\det\left(  Z_{f}\right)  \cdot\underbrace{v_{f}}_{=\left(  1-\delta
_{f^{n-1}\left(  i\right)  ,n}\right)  _{1\leq i\leq n,\ 1\leq j\leq1}}\\
&  =\det\left(  Z_{f}\right)  \cdot\left(  1-\delta_{f^{n-1}\left(  i\right)
,n}\right)  _{1\leq i\leq n,\ 1\leq j\leq1}\\
&  =\left(  \det\left(  Z_{f}\right)  \cdot\left(  1-\delta_{f^{n-1}\left(
i\right)  ,n}\right)  \right)  _{1\leq i\leq n,\ 1\leq j\leq1}.
\end{align*}
In other words, $0=\det\left(  Z_{f}\right)  \cdot\left(  1-\delta
_{f^{n-1}\left(  i\right)  ,n}\right)  $ for each $i\in\left\{  1,2,\ldots
,n\right\}  $. Applying this to $i=u$, we obtain%
\[
0=\det\left(  Z_{f}\right)  \cdot\left(  1-\underbrace{\delta_{f^{n-1}\left(
u\right)  ,n}}_{\substack{=0\\\text{(since }f^{n-1}\left(  u\right)  \neq
n\text{)}}}\right)  =\det\left(  Z_{f}\right)  \cdot1=\det\left(
Z_{f}\right)  .
\]
This proves Proposition \ref{prop.Zf.n-pot} \textbf{(b)}.
\end{proof}

\subsection{Proof of Theorem \ref{thm.chio-gen}}

Let us finally recall a particularly basic property of determinants:

\begin{lemma}
\label{lem.det.rowmult}Let $m\in\mathbb{N}$. Let $A=\left(  a_{i,j}\right)
_{1\leq i\leq m,\ 1\leq j\leq m}\in\mathbb{K}^{m\times m}$ be an $m\times
m$-matrix. Let $b_{1},b_{2},\ldots,b_{m}$ be $m$ elements of $\mathbb{K}$.
Then,%
\[
\det\left(  \left(  b_{i}a_{i,j}\right)  _{1\leq i\leq m,\ 1\leq j\leq
m}\right)  =\left(  \prod_{i=1}^{m}b_{i}\right)  \det A.
\]

\end{lemma}

(Again, see the Appendix for the proof of this lemma.)

We can now finally prove Theorem \ref{thm.chio-gen}:

\begin{proof}
[Proof of Theorem \ref{thm.chio-gen}.]The identities we want to prove (both
for part \textbf{(a)} and for part \textbf{(b)}) are polynomial identities in
the entries of $A$. Thus, we can WLOG assume that all these entries are
invertible.\footnote{Here is a more detailed justification for this
\textquotedblleft WLOG\textquotedblright:
\par
Let us restrict ourselves to Theorem \ref{thm.chio-gen} \textbf{(b)}. (The
argument for Theorem \ref{thm.chio-gen} \textbf{(a)} is analogous.)
\par
Assume that Theorem \ref{thm.chio-gen} \textbf{(b)} is proven in the case when
all entries of $A$ are invertible. We now must show that Theorem
\ref{thm.chio-gen} \textbf{(b)} always holds.
\par
Let $n$ be a positive integer such that $n\geq2$. Let $f:\left\{
1,2,\ldots,n\right\}  \rightarrow\left\{  1,2,\ldots,n\right\}  $ be an
$n$-potent map. Then, Theorem \ref{thm.chio-gen} \textbf{(b)} claims that%
\begin{equation}
\det B=\left(  \operatorname*{abut}\nolimits_{f}A\right)  \cdot\det A
\label{pf.thm.chio-gen.WLOG.5}%
\end{equation}
for every $n\times n$-matrix $A=\left(  a_{i,j}\right)  _{1\leq i\leq
n,\ 1\leq j\leq n}\in\mathbb{K}^{n\times n}$, where $B$ is as defined in
Theorem \ref{thm.chio-gen}. The equality (\ref{pf.thm.chio-gen.WLOG.5})
rewrites as
\begin{align}
&  \sum_{\sigma\in S_{n-1}}\left(  -1\right)  ^{\sigma}\prod_{i=1}%
^{n-1}\left(  a_{i,\sigma\left(  i\right)  }a_{f\left(  i\right)  ,n}%
-a_{i,n}a_{f\left(  i\right)  ,\sigma\left(  i\right)  }\right) \nonumber\\
&  =\left(  a_{n,n}^{\left\vert f^{-1}\left(  n\right)  \right\vert -2}%
\prod_{\substack{i\in\left\{  1,2,\ldots,n-1\right\}  ;\\f\left(  i\right)
\neq n}}a_{f\left(  i\right)  ,n}\right)  \cdot\sum_{\sigma\in S_{n}}\left(
-1\right)  ^{\sigma}\prod_{i=1}^{n}a_{i,\sigma\left(  i\right)  }
\label{pf.thm.chio-gen.WLOG.6}%
\end{align}
(because we have
\begin{align*}
\det\underbrace{B}_{=\left(  a_{i,j}a_{f\left(  i\right)  ,n}-a_{i,n}%
a_{f\left(  i\right)  ,j}\right)  _{1\leq i\leq n-1,\ 1\leq j\leq n-1}}  &
=\det\left(  \left(  a_{i,j}a_{f\left(  i\right)  ,n}-a_{i,n}a_{f\left(
i\right)  ,j}\right)  _{1\leq i\leq n-1,\ 1\leq j\leq n-1}\right) \\
&  =\sum_{\sigma\in S_{n-1}}\left(  -1\right)  ^{\sigma}\prod_{i=1}%
^{n-1}\left(  a_{i,\sigma\left(  i\right)  }a_{f\left(  i\right)  ,n}%
-a_{i,n}a_{f\left(  i\right)  ,\sigma\left(  i\right)  }\right)
\end{align*}
and $\operatorname*{abut}\nolimits_{f}A=a_{n,n}^{\left\vert f^{-1}\left(
n\right)  \right\vert -2}\prod_{\substack{i\in\left\{  1,2,\ldots,n-1\right\}
;\\f\left(  i\right)  \neq n}}a_{f\left(  i\right)  ,n}$ and $\det
A=\sum_{\sigma\in S_{n}}\left(  -1\right)  ^{\sigma}\prod_{i=1}^{n}%
a_{i,\sigma\left(  i\right)  }$). Thus, Theorem \ref{thm.chio-gen}
\textbf{(b)} (for our given $n$ and $f$) is equivalent to the claim that
(\ref{pf.thm.chio-gen.WLOG.6}) holds for every $n\times n$-matrix $\left(
a_{i,j}\right)  _{1\leq i\leq n,\ 1\leq j\leq n}\in\mathbb{K}^{n\times n}$.
\par
Now, let $\mathbb{P}$ be the polynomial ring $\mathbb{Z}\left[  X_{i,j}%
\ \mid\ \left(  i,j\right)  \in\left\{  1,2,\ldots,n\right\}  ^{2}\right]  $
in the $n^{2}$ indeterminates $X_{i,j}$ for $\left(  i,j\right)  \in\left\{
1,2,\ldots,n\right\}  ^{2}$. Let $\mathbb{F}$ be the quotient field of
$\mathbb{P}$; this is the field $\mathbb{Q}\left(  X_{i,j}\ \mid\ \left(
i,j\right)  \in\left\{  1,2,\ldots,n\right\}  ^{2}\right)  $ of rational
functions in the same indeterminates (but over $\mathbb{Q}$).
\par
Let $A_{X}$ be the $n\times n$-matrix $\left(  X_{i,j}\right)  _{1\leq i\leq
n,\ 1\leq j\leq n}\in\mathbb{P}^{n\times n}$. If we regard $A_{X}$ as a matrix
in $\mathbb{F}^{n\times n}$, then all entries of $A_{X}$ are invertible
(because they are nonzero elements of the field $\mathbb{F}$). Hence, Theorem
\ref{thm.chio-gen} \textbf{(b)} can be applied to $\mathbb{F}$, $A_{X}$,
$X_{i,j}$ and $B_{X}$ instead of $\mathbb{K}$, $A$, $a_{i,j}$ and $B$ (because
we have assumed that Theorem \ref{thm.chio-gen} \textbf{(b)} is proven in the
case when all entries of $A$ are invertible). As we know, this means that
(\ref{pf.thm.chio-gen.WLOG.6}) holds for $a_{i,j}=X_{i,j}$. In other words, we
have%
\begin{align}
&  \sum_{\sigma\in S_{n-1}}\left(  -1\right)  ^{\sigma}\prod_{i=1}%
^{n-1}\left(  X_{i,\sigma\left(  i\right)  }X_{f\left(  i\right)  ,n}%
-X_{i,n}X_{f\left(  i\right)  ,\sigma\left(  i\right)  }\right) \nonumber\\
&  =\left(  X_{n,n}^{\left\vert f^{-1}\left(  n\right)  \right\vert -2}%
\prod_{\substack{i\in\left\{  1,2,\ldots,n-1\right\}  ;\\f\left(  i\right)
\neq n}}X_{f\left(  i\right)  ,n}\right)  \cdot\sum_{\sigma\in S_{n}}\left(
-1\right)  ^{\sigma}\prod_{i=1}^{n}X_{i,\sigma\left(  i\right)  }.
\label{pf.thm.chio-gen.WLOG.7}%
\end{align}
\par
Now, let $\left(  a_{i,j}\right)  _{1\leq i\leq n,\ 1\leq j\leq n}%
\in\mathbb{K}^{n\times n}$ be an $n\times n$-matrix. The equality
(\ref{pf.thm.chio-gen.WLOG.7}) is an identity between polynomials in the
polynomial ring $\mathbb{P}$. Thus, we can substitute $a_{i,j}$ for each
$X_{i,j}$ in this equality. As a result, we obtain the equality
(\ref{pf.thm.chio-gen.WLOG.6}).
\par
Thus we have shown that (\ref{pf.thm.chio-gen.WLOG.6}) holds for every
$n\times n$-matrix $\left(  a_{i,j}\right)  _{1\leq i\leq n,\ 1\leq j\leq
n}\in\mathbb{K}^{n\times n}$. As we have already explained, this is just a
restatement of Theorem \ref{thm.chio-gen} \textbf{(b)}; hence, Theorem
\ref{thm.chio-gen} \textbf{(b)} is proven in full generality.
\par
(The justification above is a typical use of the \textquotedblleft method of
universal identities\textquotedblright. See \cite{Conrad09} for examples of
similar justifications, albeit used in different settings.)} In other words,
we can assume that $a_{i,j}$ is invertible for each $\left(  i,j\right)
\in\left\{  1,2,\ldots,n\right\}  ^{2}$. Assume this.

Let $C$ be the $\left(  n-1\right)  \times\left(  n-1\right)  $-matrix%
\[
\left(  \dfrac{a_{i,j}}{a_{i,n}}-\dfrac{a_{f\left(  i\right)  ,j}}{a_{f\left(
i\right)  ,n}}\right)  _{1\leq i\leq n-1,\ 1\leq j\leq n-1}\in\mathbb{K}%
^{\left(  n-1\right)  \times\left(  n-1\right)  }.
\]
Thus, Lemma \ref{lem.det.rowmult} (applied to $n-1$, $C$, $\dfrac{a_{i,j}%
}{a_{i,n}}-\dfrac{a_{f\left(  i\right)  ,j}}{a_{f\left(  i\right)  ,n}}$ and
$a_{i,n}a_{f\left(  i\right)  ,n}$ instead of $m$, $A$, $a_{i,j}$ and $b_{i}$)
yields%
\begin{align*}
&  \det\left(  \left(  a_{i,n}a_{f\left(  i\right)  ,n}\left(  \dfrac{a_{i,j}%
}{a_{i,n}}-\dfrac{a_{f\left(  i\right)  ,j}}{a_{f\left(  i\right)  ,n}%
}\right)  \right)  _{1\leq i\leq n-1,\ 1\leq j\leq n-1}\right) \\
&  =\left(  \prod_{i=1}^{n-1}\left(  a_{i,n}a_{f\left(  i\right)  ,n}\right)
\right)  \det C.
\end{align*}
Comparing this with%
\begin{align*}
&  \det\left(  \left(  \underbrace{a_{i,n}a_{f\left(  i\right)  ,n}\left(
\dfrac{a_{i,j}}{a_{i,n}}-\dfrac{a_{f\left(  i\right)  ,j}}{a_{f\left(
i\right)  ,n}}\right)  }_{=a_{i,j}a_{f\left(  i\right)  ,n}-a_{i,n}a_{f\left(
i\right)  ,j}}\right)  _{1\leq i\leq n-1,\ 1\leq j\leq n-1}\right) \\
&  =\det\left(  \underbrace{\left(  a_{i,j}a_{f\left(  i\right)  ,n}%
-a_{i,n}a_{f\left(  i\right)  ,j}\right)  _{1\leq i\leq n-1,\ 1\leq j\leq
n-1}}_{=B}\right)  =\det B,
\end{align*}
we find%
\begin{equation}
\det B=\left(  \prod_{i=1}^{n-1}\left(  a_{i,n}a_{f\left(  i\right)
,n}\right)  \right)  \det C. \label{pf.thm.chio-gen.alg.BfromC}%
\end{equation}
It remains to compute $\det C$.

For every $\left(  i,j\right)  \in\left\{  1,2,\ldots,n\right\}  ^{2}$, define
an element $d_{i,j}\in\mathbb{K}$ by%
\[
d_{i,j}=%
\begin{cases}
\dfrac{a_{i,j}}{a_{i,n}}-\dfrac{a_{f\left(  i\right)  ,j}}{a_{f\left(
i\right)  ,n}}, & \text{if }i<n;\\
\dfrac{a_{i,j}}{a_{i,n}}, & \text{if }i=n
\end{cases}
.
\]
For every $i\in\left\{  1,2,\ldots,n-1\right\}  $, the definition of $d_{i,n}$
yields%
\begin{align*}
d_{i,n}  &  =%
\begin{cases}
\dfrac{a_{i,n}}{a_{i,n}}-\dfrac{a_{f\left(  i\right)  ,n}}{a_{f\left(
i\right)  ,n}}, & \text{if }i<n;\\
\dfrac{a_{i,n}}{a_{i,n}}, & \text{if }i=n
\end{cases}
=\underbrace{\dfrac{a_{i,n}}{a_{i,n}}}_{=1}-\underbrace{\dfrac{a_{f\left(
i\right)  ,n}}{a_{f\left(  i\right)  ,n}}}_{=1}\ \ \ \ \ \ \ \ \ \ \left(
\text{since }i<n\right) \\
&  =1-1=0.
\end{align*}
Moreover, the definition of $d_{n,n}$ yields%
\begin{align*}
d_{n,n}  &  =%
\begin{cases}
\dfrac{a_{n,n}}{a_{n,n}}-\dfrac{a_{f\left(  n\right)  ,n}}{a_{f\left(
n\right)  ,n}}, & \text{if }n<n;\\
\dfrac{a_{n,n}}{a_{n,n}}, & \text{if }n=n
\end{cases}
=\dfrac{a_{n,n}}{a_{n,n}}\ \ \ \ \ \ \ \ \ \ \left(  \text{since }n=n\right)
\\
&  =1.
\end{align*}
Finally, every $i\in\left\{  1,2,\ldots,n-1\right\}  $ and $j\in\left\{
1,2,\ldots,n\right\}  $ satisfy%
\begin{equation}
d_{i,j}=%
\begin{cases}
\dfrac{a_{i,j}}{a_{i,n}}-\dfrac{a_{f\left(  i\right)  ,j}}{a_{f\left(
i\right)  ,n}}, & \text{if }i<n;\\
\dfrac{a_{i,j}}{a_{i,n}}, & \text{if }i=n
\end{cases}
=\dfrac{a_{i,j}}{a_{i,n}}-\dfrac{a_{f\left(  i\right)  ,j}}{a_{f\left(
i\right)  ,n}} \label{pf.thm.chio-gen.alg.dij3}%
\end{equation}
(since $i<n$).

Now, let $D$ be the $n\times n$-matrix
\[
\left(  d_{i,j}\right)  _{1\leq i\leq n,\ 1\leq j\leq n}\in\mathbb{K}^{n\times
n}.
\]

Recall that $d_{i,n}=0$ for every $i\in\left\{  1,2,\ldots,n-1\right\}  $.
Hence, Lemma \ref{lem.det.ain=0} (applied to $D$ and $d_{i,j}$ instead of $A$
and $a_{i,j}$) shows that%
\begin{align*}
\det D  &  =\underbrace{d_{n,n}}_{=1}\det\left(  \left(  \underbrace{d_{i,j}%
}_{\substack{=\dfrac{a_{i,j}}{a_{i,n}}-\dfrac{a_{f\left(  i\right)  ,j}%
}{a_{f\left(  i\right)  ,n}}\\\text{(by (\ref{pf.thm.chio-gen.alg.dij3}))}%
}}\right)  _{1\leq i\leq n-1,\ 1\leq j\leq n-1}\right) \\
&  =\det\underbrace{\left(  \left(  \dfrac{a_{i,j}}{a_{i,n}}-\dfrac
{a_{f\left(  i\right)  ,j}}{a_{f\left(  i\right)  ,n}}\right)  _{1\leq i\leq
n-1,\ 1\leq j\leq n-1}\right)  }_{=C}=\det C.
\end{align*}
Hence, (\ref{pf.thm.chio-gen.alg.BfromC}) becomes%
\begin{align}
\det B  &  =\left(  \prod_{i=1}^{n-1}\left(  a_{i,n}a_{f\left(  i\right)
,n}\right)  \right)  \underbrace{\det C}_{=\det D}\nonumber\\
&  =\left(  \prod_{i=1}^{n-1}\left(  a_{i,n}a_{f\left(  i\right)  ,n}\right)
\right)  \det D. \label{pf.thm.chio-gen.alg.BfromD}%
\end{align}
Hence, we only need to compute $\det D$. How do we do this?

Let $E$ be the $n\times n$-matrix $\left(  \dfrac{a_{i,j}}{a_{i,n}}\right)
_{1\leq i\leq n,\ 1\leq j\leq n}\in\mathbb{K}^{n\times n}$.

Recall that $A=\left(  a_{i,j}\right)  _{1\leq i\leq n,\ 1\leq j\leq n}$.
Lemma \ref{lem.det.rowmult} (applied to $m=n$ and $b_{i}=\dfrac{1}{a_{i,n}}$)
thus yields%
\[
\det\left(  \left(  \dfrac{1}{a_{i,n}}a_{i,j}\right)  _{1\leq i\leq n,\ 1\leq
j\leq n}\right)  =\left(  \prod_{i=1}^{n}\dfrac{1}{a_{i,n}}\right)  \det A.
\]
Compared with
\[
\det\left(  \left(  \underbrace{\dfrac{1}{a_{i,n}}a_{i,j}}_{=\dfrac{a_{i,j}%
}{a_{i,n}}}\right)  _{1\leq i\leq n,\ 1\leq j\leq n}\right)  =\det\left(
\underbrace{\left(  \dfrac{a_{i,j}}{a_{i,n}}\right)  _{1\leq i\leq n,\ 1\leq
j\leq n}}_{=E}\right)  =\det E,
\]
this yields%
\begin{equation}
\det E=\left(  \prod_{i=1}^{n}\dfrac{1}{a_{i,n}}\right)  \det A.
\label{pf.thm.chio-gen.alg.AfromE}%
\end{equation}

On the other hand, recall that we have defined an $n\times n$-matrix $Z_{f}$
in Definition \ref{def.Zf}. We now claim that%
\begin{equation}
D=Z_{f}E. \label{pf.thm.chio-gen.D=ZfE}%
\end{equation}

\textit{Proof of (\ref{pf.thm.chio-gen.D=ZfE}):} We have $Z_{f}=\left(
\delta_{i,j}-\left(  1-\delta_{i,n}\right)  \delta_{f\left(  i\right)
,j}\right)  _{1\leq i\leq n,\ 1\leq j\leq n}$ and \newline$E=\left(
\dfrac{a_{i,j}}{a_{i,n}}\right)  _{1\leq i\leq n,\ 1\leq j\leq n}$. Thus, the
definition of the product of two matrices yields%
\[
Z_{f}E=\left(  \sum_{k=1}^{n}\left(  \delta_{i,k}-\left(  1-\delta
_{i,n}\right)  \delta_{f\left(  i\right)  ,k}\right)  \dfrac{a_{k,j}}{a_{k,n}%
}\right)  _{1\leq i\leq n,\ 1\leq j\leq n}.
\]
Since every $\left(  i,j\right)  \in\left\{  1,2,\ldots,n\right\}  ^{2}$
satisfies%
\begin{align*}
&  \sum_{k=1}^{n}\left(  \delta_{i,k}-\left(  1-\delta_{i,n}\right)
\delta_{f\left(  i\right)  ,k}\right)  \dfrac{a_{k,j}}{a_{k,n}}\\
&  =\underbrace{\sum_{k=1}^{n}\delta_{i,k}\dfrac{a_{k,j}}{a_{k,n}}%
}_{\substack{=\dfrac{a_{i,j}}{a_{i,n}}\\\text{(because the factor }%
\delta_{i,k}\text{ in the sum}\\\text{kills every addend except the one for
}k=i\text{)}}}-\sum_{k=1}^{n}\underbrace{\left(  1-\delta_{i,n}\right)
\delta_{f\left(  i\right)  ,k}\dfrac{a_{k,j}}{a_{k,n}}}_{\substack{=\left(
1-\delta_{i,n}\right)  \dfrac{a_{f\left(  i\right)  ,j}}{a_{f\left(  i\right)
,n}}\\\text{(because the factor }\delta_{f\left(  i\right)  ,k}\text{ in the
sum}\\\text{kills every addend except the one for }k=f\left(  i\right)
\text{)}}}\\
&  =\dfrac{a_{i,j}}{a_{i,n}}-\underbrace{\left(  1-\delta_{i,n}\right)  }_{=%
\begin{cases}
1, & \text{if }i<n;\\
0, & \text{if }i=n
\end{cases}
}\dfrac{a_{f\left(  i\right)  ,j}}{a_{f\left(  i\right)  ,n}}=\dfrac{a_{i,j}%
}{a_{i,n}}-%
\begin{cases}
1, & \text{if }i<n;\\
0, & \text{if }i=n
\end{cases}
\dfrac{a_{f\left(  i\right)  ,j}}{a_{f\left(  i\right)  ,n}}\\
&  =%
\begin{cases}
\dfrac{a_{i,j}}{a_{i,n}}-\dfrac{a_{f\left(  i\right)  ,j}}{a_{f\left(
i\right)  ,n}}, & \text{if }i<n;\\
\dfrac{a_{i,j}}{a_{i,n}}, & \text{if }i=n
\end{cases}
=d_{i,j}\ \ \ \ \ \ \ \ \ \ \left(  \text{by the definition of }%
d_{i,j}\right)  ,
\end{align*}
this rewrites as%
\[
Z_{f}E=\left(  d_{i,j}\right)  _{1\leq i\leq n,\ 1\leq j\leq n}.
\]
Comparing this with $D=\left(  d_{i,j}\right)  _{1\leq i\leq n,\ 1\leq j\leq
n}$, we obtain $D=Z_{f}E$. This proves (\ref{pf.thm.chio-gen.D=ZfE}).

Now, we can prove parts \textbf{(a)} and \textbf{(b)} of Theorem
\ref{thm.chio-gen}:

\textbf{(a)} Assume that the map $f$ is not $n$-potent. Taking determinants on
both sides of (\ref{pf.thm.chio-gen.D=ZfE}), we obtain%
\[
\det D=\det\left(  Z_{f}E\right)  =\underbrace{\det\left(  Z_{f}\right)
}_{\substack{=0\\\text{(by Proposition \ref{prop.Zf.n-pot} \textbf{(b)})}%
}}\cdot\det E=0.
\]
Thus, (\ref{pf.thm.chio-gen.alg.BfromD}) becomes%
\[
\det B=\left(  \prod_{i=1}^{n-1}\left(  a_{i,n}a_{f\left(  i\right)
,n}\right)  \right)  \underbrace{\det D}_{=0}=0.
\]
This proves Theorem \ref{thm.chio-gen} \textbf{(a)}.

\textbf{(b)} Assume that the map $f$ is $n$-potent. Taking determinants on
both sides of (\ref{pf.thm.chio-gen.D=ZfE}), we obtain%
\begin{align*}
\det D  &  =\det\left(  Z_{f}E\right)  =\underbrace{\det\left(  Z_{f}\right)
}_{\substack{=1\\\text{(by Proposition \ref{prop.Zf.n-pot} \textbf{(a)})}%
}}\cdot\det E=\det E\\
&  =\underbrace{\left(  \prod_{i=1}^{n}\dfrac{1}{a_{i,n}}\right)  }_{=\left(
\prod_{i=1}^{n-1}\dfrac{1}{a_{i,n}}\right)  \cdot\dfrac{1}{a_{n,n}}}\det
A\ \ \ \ \ \ \ \ \ \ \left(  \text{by (\ref{pf.thm.chio-gen.alg.AfromE}%
)}\right) \\
&  =\left(  \prod_{i=1}^{n-1}\dfrac{1}{a_{i,n}}\right)  \cdot\dfrac{1}%
{a_{n,n}}\det A.
\end{align*}
Thus, (\ref{pf.thm.chio-gen.alg.BfromD}) becomes%
\begin{align*}
\det B  &  =\left(  \prod_{i=1}^{n-1}\left(  a_{i,n}a_{f\left(  i\right)
,n}\right)  \right)  \underbrace{\det D}_{=\left(  \prod_{i=1}^{n-1}\dfrac
{1}{a_{i,n}}\right)  \cdot\dfrac{1}{a_{n,n}}\det A}\\
&  =\underbrace{\left(  \prod_{i=1}^{n-1}\left(  a_{i,n}a_{f\left(  i\right)
,n}\right)  \right)  \left(  \prod\limits_{i=1}^{n-1}\dfrac{1}{a_{i,n}%
}\right)  }_{=\prod\limits_{i=1}^{n-1}a_{f\left(  i\right)  ,n}=\prod
\limits_{i\in\left\{  1,2,\ldots,n-1\right\}  }a_{f\left(  i\right)  ,n}}%
\cdot\dfrac{1}{a_{n,n}}\det A\\
&  =\underbrace{\left(  \prod\limits_{i\in\left\{  1,2,\ldots,n-1\right\}
}a_{f\left(  i\right)  ,n}\right)  \cdot\dfrac{1}{a_{n,n}}}_{\substack{=\dfrac
{1}{a_{n,n}}\prod\limits_{i\in\left\{  1,2,\ldots,n-1\right\}  }a_{f\left(
i\right)  ,n}=\operatorname*{abut}\nolimits_{f}A\\\text{(by Remark
\ref{rmk.n-potent.abut} \textbf{(a)})}}}\det A=\left(  \operatorname*{abut}%
\nolimits_{f}A\right)  \det A.
\end{align*}
This proves Theorem \ref{thm.chio-gen} \textbf{(b)}.
\end{proof}

\subsection{Further questions}

The above -- rather indirect -- road to the matrix-tree theorem suggests the
following two questions:

\begin{itemize}
\item Is there a combinatorial proof of Theorem \ref{thm.chio-gen}? Or, at
least, is there a \textquotedblleft division-free\textquotedblright\ proof
(i.e., a proof that does not use a WLOG assumption that some of the $a_{i,j}$
are invertible or a similar trick)?

\item Can we similarly obtain some of the various generalizations and variants
of the matrix-tree theorem, such as the all-minors matrix-tree theorem
(\cite[(2)]{Chaiken82} and \cite[Theorem 6]{Sahi13})?
\end{itemize}

\section{\label{sect.app}Appendix: some standard proofs}

For the sake of completeness, let us give some proofs of standard results that
have been used without proof above.

\begin{proof}
[Proof of Remark \ref{rmk.n-potent.atleast1}.]\textbf{(a)} We have $1\neq n$
(since $n\geq2$). But the map $f$ is $n$-potent. Thus, there exists some
$k\in\mathbb{N}$ such that $f^{k}\left(  1\right)  =n$. Let $h$ be the
smallest such $k$. Then, $f^{h}\left(  1\right)  =n$. Hence, $h\neq0$ (since
$f^{h}\left(  1\right)  =n\neq1=f^{0}\left(  1\right)  $). Therefore,
$h-1\in\mathbb{N}$, so that $f^{h-1}\left(  1\right)  \neq n$ (because $h$ is
the \textbf{smallest} $k\in\mathbb{N}$ such that $f^{k}\left(  1\right)  =n$).
Hence, $f^{h-1}\left(  1\right)  \in\left\{  1,2,\ldots,n-1\right\}  $. Thus,
$f^{h-1}\left(  1\right)  $ is a $g\in\left\{  1,2,\ldots,n-1\right\}  $ such
that $f\left(  g\right)  =n$ (since $f\left(  f^{h-1}\left(  1\right)
\right)  =f^{h}\left(  1\right)  =n$). Therefore, such a $g$ exists. This
proves Remark \ref{rmk.n-potent.atleast1} \textbf{(a)}.

\textbf{(b)} The map $f$ is $n$-potent; thus, $f\left(  n\right)  =n$. Hence,
$n\in f^{-1}\left(  n\right)  $. Remark \ref{rmk.n-potent.atleast1}
\textbf{(a)} shows that there exists some $g\in\left\{  1,2,\ldots
,n-1\right\}  $ such that $f\left(  g\right)  =n$. Consider this $g$. From
$f\left(  g\right)  =n$, we obtain $g\in f^{-1}\left(  n\right)  $. From
$g\in\left\{  1,2,\ldots,n-1\right\}  $, we obtain $g\neq n$. Hence, $g$ and
$n$ are two distinct elements of the set $f^{-1}\left(  n\right)  $.
Consequently, $\left\vert f^{-1}\left(  n\right)  \right\vert \geq2$. This
proves Remark \ref{rmk.n-potent.atleast1} \textbf{(b)}.
\end{proof}

\begin{proof}
[Proof of Remark \ref{rmk.n-potent.abut}.]\textbf{(b)} We have $n\in
f^{-1}\left(  n\right)  $ (since $f\left(  n\right)  =n$) and $g\in
f^{-1}\left(  n\right)  $ (since $f\left(  g\right)  =n$). Moreover, $g\neq n$
(since $g\in\left\{  1,2,\ldots,n-1\right\}  $). Hence, $g$ and $n$ are two
distinct elements of $f^{-1}\left(  n\right)  $. Hence, $\left\vert
f^{-1}\left(  n\right)  \setminus\left\{  n,g\right\}  \right\vert =\left\vert
f^{-1}\left(  n\right)  \right\vert -2$. But%
\begin{align*}
&  \left\{  i\in\left\{  1,2,\ldots,n-1\right\}  \setminus\left\{  g\right\}
\ \mid\ f\left(  i\right)  =n\right\} \\
&  =f^{-1}\left(  n\right)  \cap\left(  \left\{  1,2,\ldots,n-1\right\}
\setminus\left\{  g\right\}  \right) \\
&  =\underbrace{f^{-1}\left(  n\right)  \cap\left\{  1,2,\ldots,n-1\right\}
}_{=f^{-1}\left(  n\right)  \setminus\left\{  n\right\}  }\setminus\left\{
g\right\}  =\left(  f^{-1}\left(  n\right)  \setminus\left\{  n\right\}
\right)  \setminus\left\{  g\right\} \\
&  =f^{-1}\left(  n\right)  \setminus\left\{  n,g\right\}
\end{align*}
so that%
\begin{equation}
\left\vert \left\{  i\in\left\{  1,2,\ldots,n-1\right\}  \setminus\left\{
g\right\}  \ \mid\ f\left(  i\right)  =n\right\}  \right\vert =\left\vert
f^{-1}\left(  n\right)  \setminus\left\{  n,g\right\}  \right\vert =\left\vert
f^{-1}\left(  n\right)  \right\vert -2. \label{pf.rmk.n-potent.abut.a.1}%
\end{equation}
Now,%
\begin{align*}
&  \prod_{\substack{i\in\left\{  1,2,\ldots,n-1\right\}  ;\\i\neq
g}}a_{f\left(  i\right)  ,n}\\
&  =\prod_{i\in\left\{  1,2,\ldots,n-1\right\}  \setminus\left\{  g\right\}
}a_{f\left(  i\right)  ,n}=\left(  \prod_{\substack{i\in\left\{
1,2,\ldots,n-1\right\}  \setminus\left\{  g\right\}  ;\\f\left(  i\right)
=n}}\underbrace{a_{f\left(  i\right)  ,n}}_{\substack{=a_{n,n}\\\text{(since
}f\left(  i\right)  =n\text{)}}}\right)  \left(  \prod_{\substack{i\in\left\{
1,2,\ldots,n-1\right\}  \setminus\left\{  g\right\}  ;\\f\left(  i\right)
\neq n}}a_{f\left(  i\right)  ,n}\right) \\
&  =\underbrace{\left(  \prod_{\substack{i\in\left\{  1,2,\ldots,n-1\right\}
\setminus\left\{  g\right\}  ;\\f\left(  i\right)  =n}}a_{n,n}\right)
}_{\substack{=a_{n,n}^{\left\vert \left\{  i\in\left\{  1,2,\ldots
,n-1\right\}  \setminus\left\{  g\right\}  \ \mid\ f\left(  i\right)
=n\right\}  \right\vert }=a_{n,n}^{\left\vert f^{-1}\left(  n\right)
\right\vert -2}\\\text{(by (\ref{pf.rmk.n-potent.abut.a.1}))}}}\left(
\prod_{\substack{i\in\left\{  1,2,\ldots,n-1\right\}  \setminus\left\{
g\right\}  ;\\f\left(  i\right)  \neq n}}a_{f\left(  i\right)  ,n}\right) \\
&  =a_{n,n}^{\left\vert f^{-1}\left(  n\right)  \right\vert -2}\left(
\prod_{\substack{i\in\left\{  1,2,\ldots,n-1\right\}  ;\\f\left(  i\right)
\neq n}}a_{f\left(  i\right)  ,n}\right)  =\operatorname*{abut}\nolimits_{f}A
\end{align*}
(by the definition of $\operatorname*{abut}\nolimits_{f}A$). This proves
Remark \ref{rmk.n-potent.abut} \textbf{(b)}.

\textbf{(a)} Assume that $a_{n,n}\in\mathbb{K}$ is invertible. Fix
$g\in\left\{  1,2,\ldots,n-1\right\}  $ as in Remark \ref{rmk.n-potent.abut}
\textbf{(b)}. Then,%
\[
\prod_{i\in\left\{  1,2,\ldots,n-1\right\}  }a_{f\left(  i\right)
,n}=\underbrace{a_{f\left(  g\right)  ,n}}_{\substack{=a_{n,n}\\\text{(since
}f\left(  g\right)  =n\text{)}}}\underbrace{\prod_{\substack{i\in\left\{
1,2,\ldots,n-1\right\}  ;\\i\neq g}}a_{f\left(  i\right)  ,n}}%
_{\substack{=\operatorname*{abut}\nolimits_{f}A\\\text{(by Remark
\ref{rmk.n-potent.abut} \textbf{(b)})}}}=a_{n,n}\operatorname*{abut}%
\nolimits_{f}A,
\]
so that $\operatorname*{abut}\nolimits_{f}A=\dfrac{1}{a_{n,n}}\prod
_{i\in\left\{  1,2,\ldots,n-1\right\}  }a_{f\left(  i\right)  ,n}$. This
proves Remark \ref{rmk.n-potent.abut} \textbf{(a)}.
\end{proof}

\begin{proof}
[Proof of Lemma \ref{lem.supergen.sum1}.]We have $G=\left(  \sum_{k=1}%
^{n}b_{i,k}d_{i,j,k}\right)  _{1\leq i\leq m,\ 1\leq j\leq m}$. Thus, the
definition of a determinant yields%
\begin{align*}
\det G  &  =\sum_{\sigma\in S_{m}}\left(  -1\right)  ^{\sigma}%
\underbrace{\prod_{i=1}^{m}\left(  \sum_{k=1}^{n}b_{i,k}d_{i,\sigma\left(
i\right)  ,k}\right)  }_{\substack{=\sum_{f:\left\{  1,2,\ldots,m\right\}
\rightarrow\left\{  1,2,\ldots,n\right\}  }\prod_{i=1}^{m}\left(
b_{i,f\left(  i\right)  }d_{i,\sigma\left(  i\right)  ,f\left(  i\right)
}\right)  \\\text{(by the product rule)}}}\\
&  =\sum_{\sigma\in S_{m}}\left(  -1\right)  ^{\sigma}\sum_{f:\left\{
1,2,\ldots,m\right\}  \rightarrow\left\{  1,2,\ldots,n\right\}  }\prod
_{i=1}^{m}\left(  b_{i,f\left(  i\right)  }d_{i,\sigma\left(  i\right)
,f\left(  i\right)  }\right) \\
&  =\sum_{f:\left\{  1,2,\ldots,m\right\}  \rightarrow\left\{  1,2,\ldots
,n\right\}  }\sum_{\sigma\in S_{m}}\left(  -1\right)  ^{\sigma}%
\underbrace{\prod_{i=1}^{m}\left(  b_{i,f\left(  i\right)  }d_{i,\sigma\left(
i\right)  ,f\left(  i\right)  }\right)  }_{=\left(  \prod_{i=1}^{m}%
b_{i,f\left(  i\right)  }\right)  \left(  \prod_{i=1}^{m}d_{i,\sigma\left(
i\right)  ,f\left(  i\right)  }\right)  }\\
&  =\sum_{f:\left\{  1,2,\ldots,m\right\}  \rightarrow\left\{  1,2,\ldots
,n\right\}  }\left(  \prod_{i=1}^{m}b_{i,f\left(  i\right)  }\right)
\underbrace{\sum_{\sigma\in S_{m}}\left(  -1\right)  ^{\sigma}\left(
\prod_{i=1}^{m}d_{i,\sigma\left(  i\right)  ,f\left(  i\right)  }\right)
}_{\substack{=\det\left(  \left(  d_{i,j,f\left(  i\right)  }\right)  _{1\leq
i\leq m,\ 1\leq j\leq m}\right)  \\\text{(by the definition of a
determinant)}}}\\
&  =\sum_{f:\left\{  1,2,\ldots,m\right\}  \rightarrow\left\{  1,2,\ldots
,n\right\}  }\left(  \prod_{i=1}^{m}b_{i,f\left(  i\right)  }\right)
\det\left(  \left(  d_{i,j,f\left(  i\right)  }\right)  _{1\leq i\leq
m,\ 1\leq j\leq m}\right)  .
\end{align*}

\end{proof}

\begin{proof}
[Proof of Proposition \ref{prop.map.image}.]The elements $f^{0}\left(
i\right)  ,f^{1}\left(  i\right)  ,\ldots,f^{n}\left(  i\right)  $ are $n+1$
elements of the $n$-element set $\left\{  1,2,\ldots,n\right\}  $. Thus, by
the pigeonhole principle, we see that two of these elements must be equal. In
other words, there exist two elements $u$ and $v$ of $\left\{  0,1,\ldots
,n\right\}  $ such that $u<v$ and $f^{u}\left(  i\right)  =f^{v}\left(
i\right)  $. Consider these $u$ and $v$. We have $v\in\left\{  0,1,\ldots
,n\right\}  $, so that $v\leq n$ and thus $v-1\leq n-1$. Hence, $\left\{
0,1,\ldots,v-1\right\}  \subseteq\left\{  0,1,\ldots,n-1\right\}  $.

We have $u<v$, so that $u\leq v-1$ (since $u$ and $v$ are integers). Thus,
$u\in\left\{  0,1,\ldots,v-1\right\}  $ (since $u$ is a nonnegative integer).
Hence, $0\leq u\leq v-1$, so that $0\in\left\{  0,1,\ldots,v-1\right\}  $.

Let $S$ be the set $\left\{  f^{0}\left(  i\right)  ,f^{1}\left(  i\right)
,\ldots,f^{v-1}\left(  i\right)  \right\}  $. From $u\in\left\{
0,1,\ldots,v-1\right\}  $, we obtain $f^{u}\left(  i\right)  \in\left\{
f^{0}\left(  i\right)  ,f^{1}\left(  i\right)  ,\ldots,f^{v-1}\left(
i\right)  \right\}  =S$. From $0\in\left\{  0,1,\ldots,v-1\right\}  $, we
obtain $f^{0}\left(  i\right)  \in\left\{  f^{0}\left(  i\right)
,f^{1}\left(  i\right)  ,\ldots,f^{v-1}\left(  i\right)  \right\}  =S$.

Now,
\begin{equation}
f\left(  s\right)  \in S\ \ \ \ \ \ \ \ \ \ \text{for every }s\in S
\label{pf.prop.map.image.1}%
\end{equation}
\footnote{\textit{Proof of (\ref{pf.prop.map.image.1}):} Let $s\in S$.
\par
We have $s\in S=\left\{  f^{0}\left(  i\right)  ,f^{1}\left(  i\right)
,\ldots,f^{v-1}\left(  i\right)  \right\}  $. In other words, $s=f^{h}\left(
i\right)  $ for some $h\in\left\{  0,1,\ldots,v-1\right\}  $. Consider this
$h$. Thus, $f\left(  \underbrace{s}_{=f^{h}\left(  i\right)  }\right)
=f\left(  f^{h}\left(  i\right)  \right)  =f^{h+1}\left(  i\right)  $.
\par
We want to prove that $f\left(  s\right)  \in S$. We are in one of the
following two cases:
\par
\textit{Case 1:} We have $h=v-1$.
\par
\textit{Case 2:} We have $h\neq v-1$.
\par
Let us first consider Case 1. In this case, we have $h=v-1$. Hence, $h+1=v$.
Now, $f\left(  s\right)  =f^{h+1}\left(  i\right)  =f^{v}\left(  i\right)  $
(since $h+1=v$). Compared with $f^{u}\left(  i\right)  =f^{v}\left(  i\right)
$, this yields $f\left(  s\right)  =f^{u}\left(  i\right)  \in S$. Hence,
$f\left(  s\right)  \in S$ is proven in Case 1.
\par
Let us now consider Case 2. In this case, we have $h\neq v-1$. Combined with
$h\in\left\{  0,1,\ldots,v-1\right\}  $, this yields $h\in\left\{
0,1,\ldots,v-1\right\}  \setminus\left\{  v-1\right\}  =\left\{
0,1,\ldots,\left(  v-1\right)  -1\right\}  $, so that $h+1\in\left\{
0,1,\ldots,v-1\right\}  $. Thus, $f^{h+1}\left(  i\right)  \in\left\{
f^{0}\left(  i\right)  ,f^{1}\left(  i\right)  ,\ldots,f^{v-1}\left(
i\right)  \right\}  =S$. Hence, $f\left(  s\right)  =f^{h+1}\left(  i\right)
\in S$. Thus, $f\left(  s\right)  \in S$ is proven in Case 2.
\par
We have now proven $f\left(  s\right)  \in S$ in each of the two Cases 1 and
2. Thus, $f\left(  s\right)  \in S$ always holds. This proves
(\ref{pf.prop.map.image.1}).}.

Now, we can easily see that
\begin{equation}
f^{k}\left(  i\right)  \in S\ \ \ \ \ \ \ \ \ \ \text{for every }%
k\in\mathbb{N} \label{pf.prop.map.image.2}%
\end{equation}
\footnote{\textit{Proof of (\ref{pf.prop.map.image.2}):} We shall prove
(\ref{pf.prop.map.image.2}) by induction over $k$:
\par
\textit{Induction base:} We have $f^{0}\left(  i\right)  \in S$. In other
words, (\ref{pf.prop.map.image.2}) holds for $k=0$. This completes the
induction base.
\par
\textit{Induction step:} Let $K\in\mathbb{N}$. Assume that
(\ref{pf.prop.map.image.2}) holds for $k=K$. We must prove that
(\ref{pf.prop.map.image.2}) holds for $k=K+1$.
\par
We have assumed that (\ref{pf.prop.map.image.2}) holds for $k=K$. In other
words, $f^{K}\left(  i\right)  \in S$. Thus, (\ref{pf.prop.map.image.1})
(applied to $s=f^{K}\left(  i\right)  $) yields $f\left(  f^{K}\left(
i\right)  \right)  \in S$. Thus, $f^{K+1}\left(  i\right)  =f\left(
f^{K}\left(  i\right)  \right)  \in S$. In other words,
(\ref{pf.prop.map.image.2}) holds for $k=K+1$. This completes the induction
step. Hence, (\ref{pf.prop.map.image.2}) is proven by induction.}.

On the other hand,
\begin{align*}
S  &  =\left\{  f^{0}\left(  i\right)  ,f^{1}\left(  i\right)  ,\ldots
,f^{v-1}\left(  i\right)  \right\}  =\left\{  f^{s}\left(  i\right)
\ \mid\ s\in\left\{  0,1,\ldots,v-1\right\}  \right\} \\
&  \subseteq\left\{  f^{s}\left(  i\right)  \ \mid\ s\in\left\{
0,1,\ldots,n-1\right\}  \right\}  \ \ \ \ \ \ \ \ \ \ \left(  \text{since
}\left\{  0,1,\ldots,v-1\right\}  \subseteq\left\{  0,1,\ldots,n-1\right\}
\right)  .
\end{align*}
Hence, for every $k\in\mathbb{N}$, we have%
\begin{align*}
f^{k}\left(  i\right)   &  \in S\ \ \ \ \ \ \ \ \ \ \left(  \text{by
(\ref{pf.prop.map.image.2})}\right) \\
&  \subseteq\left\{  f^{s}\left(  i\right)  \ \mid\ s\in\left\{
0,1,\ldots,n-1\right\}  \right\}  .
\end{align*}
This proves Proposition \ref{prop.map.image}.
\end{proof}

\begin{proof}
[Proof of Proposition \ref{prop.map.preim}.]$\Longrightarrow:$ Assume that
$f^{n-1}\left(  i\right)  =n$. Thus, there exists some $k\in\mathbb{N}$ such
that $f^{k}\left(  i\right)  =n$ (namely, $k=n-1$). This proves the
$\Longrightarrow$ direction of Proposition \ref{prop.map.preim}.

$\Longleftarrow:$ Assume that there exists some $k\in\mathbb{N}$ such that
$f^{k}\left(  i\right)  =n$. Consider this $k$. We must show that
$f^{n-1}\left(  i\right)  =n$.

We have $n=f^{k}\left(  i\right)  \in\left\{  f^{s}\left(  i\right)
\ \mid\ s\in\left\{  0,1,\ldots,n-1\right\}  \right\}  $ (by Proposition
\ref{prop.map.image}). In other words, $n=f^{s}\left(  i\right)  $ for some
$s\in\left\{  0,1,\ldots,n-1\right\}  $. Consider this $s$.

We have $f^{s}\left(  i\right)  =n$. Using this fact (and the fact that
$f\left(  n\right)  =n$), we can prove (by induction over $h$) that%
\begin{equation}
f^{h}\left(  i\right)  =n\ \ \ \ \ \ \ \ \ \ \text{for every integer }h\geq s.
\label{pf.prop.n-potent.calc.1}%
\end{equation}

But $s\in\left\{  0,1,\ldots,n-1\right\}  $, so that $s\leq n-1$ and therefore
$n-1\geq s$. Hence, (\ref{pf.prop.n-potent.calc.1}) (applied to $h=n-1$)
yields $f^{n-1}\left(  i\right)  =n$. This proves the $\Longleftarrow$
direction of Proposition \ref{prop.map.preim}.
\end{proof}

\begin{proof}
[Proof of Proposition \ref{prop.n-potent.calc}.]$\Longleftarrow:$ Assume that
$f^{n-1}\left(  \left\{  1,2,\ldots,n\right\}  \right)  =\left\{  n\right\}
$. For every $i\in\left\{  1,2,\ldots,n\right\}  $, we have%
\[
f^{n-1}\left(  \underbrace{i}_{\in\left\{  1,2,\ldots,n\right\}  }\right)  \in
f^{n-1}\left(  \left\{  1,2,\ldots,n\right\}  \right)  =\left\{  n\right\}
\]
and thus $f^{n-1}\left(  i\right)  =n$. Hence, for every $i\in\left\{
1,2,\ldots,n\right\}  $, there exists some $k\in\mathbb{N}$ such that
$f^{k}\left(  i\right)  =n$ (namely, $k=n-1$). In other words, the map $f$ is
$n$-potent. This proves the $\Longleftarrow$ direction of Proposition
\ref{prop.n-potent.calc}.

$\Longrightarrow:$ Assume that the map $f$ is $n$-potent. Let $i\in\left\{
1,2,\ldots,n\right\}  $. Then, there exists some $k\in\mathbb{N}$ such that
$f^{k}\left(  i\right)  =n$ (since $f$ is $n$-potent). Thus, $f^{n-1}\left(
i\right)  =n$ (by the $\Longleftarrow$ direction of Proposition
\ref{prop.map.preim}).

Now, forget that we fixed $i$. We thus have shown that $f^{n-1}\left(
i\right)  =n$ for each $i\in\left\{  1,2,\ldots,n\right\}  $. Hence,%
\[
\left\{  f^{n-1}\left(  1\right)  ,f^{n-1}\left(  2\right)  ,\ldots
,f^{n-1}\left(  n\right)  \right\}  =\left\{  \underbrace{n,n,\ldots
,n}_{n\text{ times }n}\right\}  =\left\{  n\right\}  .
\]
Thus, $f^{n-1}\left(  \left\{  1,2,\ldots,n\right\}  \right)  =\left\{
f^{n-1}\left(  1\right)  ,f^{n-1}\left(  2\right)  ,\ldots,f^{n-1}\left(
n\right)  \right\}  =\left\{  n\right\}  $. This proves the $\Longrightarrow$
direction of Proposition \ref{prop.n-potent.calc}.
\end{proof}

\begin{proof}
[Proof of Corollary \ref{cor.n-potent.delta}.]We are in one of the following
two cases:

\textit{Case 1:} We have $f^{n-1}\left(  i\right)  =n$.

\textit{Case 2:} We have $f^{n-1}\left(  i\right)  \neq n$.

Let us consider Case 1 first. In this case, we have $f^{n-1}\left(  i\right)
=n$. Thus, $\delta_{f^{n-1}\left(  i\right)  ,n}=1$. But $f^{n}\left(
i\right)  =f\left(  \underbrace{f^{n-1}\left(  i\right)  }_{=n}\right)
=f\left(  n\right)  =n$, so that $\delta_{f^{n}\left(  i\right)  ,n}=1$.
Hence, $\delta_{f^{n-1}\left(  i\right)  ,n}=1=\delta_{f^{n}\left(  i\right)
,n}$. Thus, Corollary \ref{cor.n-potent.delta} is proven in Case 1.

Let us now consider Case 2. In this case, we have $f^{n-1}\left(  i\right)
\neq n$. Thus, $\delta_{f^{n-1}\left(  i\right)  ,n}=0$. On the other hand, we
have $f^{n}\left(  i\right)  \neq n$\ \ \ \ \footnote{\textit{Proof.} Assume
the contrary. Thus, $f^{n}\left(  i\right)  =n$. Hence, there exists some
$k\in\mathbb{N}$ such that $f^{k}\left(  i\right)  =n$ (namely, $k=n$). Thus,
$f^{n-1}\left(  i\right)  =n$ (according to the $\Longleftarrow$ direction of
Proposition \ref{prop.map.preim}). This contradicts $f^{n-1}\left(  i\right)
\neq n$. This contradiction proves that our assumption was wrong, qed.}.
Hence, $\delta_{f^{n}\left(  i\right)  ,n}=0$. Hence, $\delta_{f^{n-1}\left(
i\right)  ,n}=0=\delta_{f^{n}\left(  i\right)  ,n}$. Thus, Corollary
\ref{cor.n-potent.delta} is proven in Case 2.

Now, we have proven Corollary \ref{cor.n-potent.delta} in each of the two
Cases 1 and 2. Hence, Corollary \ref{cor.n-potent.delta} always holds.
\end{proof}

\begin{proof}
[Proof of Lemma \ref{lem.det.rowmult}.]The definition of $\det A$ yields $\det
A=\sum_{\sigma\in S_{m}}\left(  -1\right)  ^{\sigma}\prod_{i=1}^{m}%
a_{i,\sigma\left(  i\right)  }$ (since $A=\left(  a_{i,j}\right)  _{1\leq
i\leq m,\ 1\leq j\leq m}$). On the other hand, the definition of $\det\left(
\left(  b_{i}a_{i,j}\right)  _{1\leq i\leq m,\ 1\leq j\leq m}\right)  $ yields%
\begin{align*}
\det\left(  \left(  b_{i}a_{i,j}\right)  _{1\leq i\leq m,\ 1\leq j\leq
m}\right)   &  =\sum_{\sigma\in S_{m}}\left(  -1\right)  ^{\sigma
}\underbrace{\prod_{i=1}^{m}\left(  b_{i}a_{i,\sigma\left(  i\right)
}\right)  }_{=\left(  \prod_{i=1}^{m}b_{i}\right)  \left(  \prod_{i=1}%
^{m}a_{i,\sigma\left(  i\right)  }\right)  }\\
&  =\sum_{\sigma\in S_{m}}\left(  -1\right)  ^{\sigma}\left(  \prod_{i=1}%
^{m}b_{i}\right)  \left(  \prod_{i=1}^{m}a_{i,\sigma\left(  i\right)  }\right)
\\
&  =\left(  \prod_{i=1}^{m}b_{i}\right)  \underbrace{\sum_{\sigma\in S_{m}%
}\left(  -1\right)  ^{\sigma}\prod_{i=1}^{m}a_{i,\sigma\left(  i\right)  }%
}_{=\det A}=\left(  \prod_{i=1}^{m}b_{i}\right)  \det A.
\end{align*}
This proves Lemma \ref{lem.det.rowmult}.
\end{proof}

\begin{verlong}
\begin{proof}
[Second proof of Lemma \ref{lem.det.rowmult}.]The matrix $\left(  b_{i}%
a_{i,j}\right)  _{1\leq i\leq m,\ 1\leq j\leq m}$ is obtained from the matrix
$A$ by multiplying the $i$-th row by $b_{i}$ for each $i\in\left\{
1,2,\ldots,m\right\}  $. Hence, its determinant is obtained from the
determinant of $A$ by multiplying with $b_{1}$, with $b_{2}$, and so on. In
other words,%
\[
\det\left(  \left(  b_{i}a_{i,j}\right)  _{1\leq i\leq m,\ 1\leq j\leq
m}\right)  =\left(  \prod_{i=1}^{m}b_{i}\right)  \det A.
\]
This proves Lemma \ref{lem.det.rowmult}.
\end{proof}

\begin{proof}
[Third proof of Lemma \ref{lem.det.rowmult}.]We have $b_{i}a_{i,j}=\sum
_{k=1}^{1}b_{i}a_{i,j}$ (since $\sum_{k=1}^{1}b_{i}a_{i,j}=b_{i}a_{i,j}$) for
every $\left(  i,j\right)  \in\left\{  1,2,\ldots,m\right\}  ^{2}$. Thus,
$\left(  b_{i}a_{i,j}\right)  _{1\leq i\leq m,\ 1\leq j\leq m}=\left(
\sum_{k=1}^{1}b_{i}a_{i,j}\right)  _{1\leq i\leq m,\ 1\leq j\leq m}$. Thus,
Lemma \ref{lem.supergen.sum1} (applied to $n=1$, $b_{i,k}=b_{i}$,
$d_{i,j,k}=a_{i,j}$ and $G=\left(  b_{i}a_{i,j}\right)  _{1\leq i\leq
m,\ 1\leq j\leq m}$) yields%
\begin{align*}
&  \det\left(  \left(  b_{i}a_{i,j}\right)  _{1\leq i\leq m,\ 1\leq j\leq
m}\right) \\
&  =\sum_{f:\left\{  1,2,\ldots,m\right\}  \rightarrow\left\{  1,2,\ldots
,1\right\}  }\left(  \prod_{i=1}^{m}b_{i}\right)  \det\left(  \left(
a_{i,j}\right)  _{1\leq i\leq m,\ 1\leq j\leq m}\right) \\
&  =\left(  \prod_{i=1}^{m}b_{i}\right)  \det\underbrace{\left(  \left(
a_{i,j}\right)  _{1\leq i\leq m,\ 1\leq j\leq m}\right)  }_{=A}\\
&  \ \ \ \ \ \ \ \ \ \ \left(  \text{since there exists precisely one map
}f:\left\{  1,2,\ldots,m\right\}  \rightarrow\left\{  1,2,\ldots,1\right\}
\right) \\
&  =\left(  \prod_{i=1}^{m}b_{i}\right)  \det A.
\end{align*}
This proves Lemma \ref{lem.det.rowmult}.
\end{proof}
\end{verlong}


\begin{thebibliography}{999999999}                                                                                        %


\bibitem[Abeles14]{Abeles}%
\href{https://doi.org/10.1016/j.laa.2014.04.010}{Francine F. Abeles,
\textit{Chi\`{o}'s and Dodgson's determinantal identities}, Linear Algebra and
its Applications, Volume 454, 1 August 2014, pp. 130--137}.

\bibitem[BerBru08]{BerBru08}%
\href{https://web.archive.org/web/20170705124943/http://www.math.ualberta.ca/ijiss/SS-Volume-4-2008/No-1-08/SS-08-01-01.pdf}{Adam
Berliner and Richard A. Brualdi, \textit{A combinatorial proof of the
Dodgson/Muir determinantal identity}, International Journal of Information and
Systems Sciences, Volume 4 (2008), Number 1, pp. 1--7.}

\bibitem[Chaiken82]{Chaiken82}%
\href{http://citeseerx.ist.psu.edu/viewdoc/summary?doi=10.1.1.363.8820}{Seth
Chaiken, \textit{A combinatorial proof of the all minors matrix tree theorem},
SIAM J. Alg. Disc. Math., Vol. 3, No. 3, September 1982, pp. 319--329.}

\bibitem[Conrad09]{Conrad09}Keith Conrad, \textit{Universal identities}, 12
October 2009.\newline\url{http://www.math.uconn.edu/~kconrad/blurbs/linmultialg/univid.pdf}

\bibitem[Eves68]{Eves}Howard Eves, \textit{Elementary Matrix Theory}, Allyn \&
Bacon, 2nd printing 1968.

\bibitem[Grinbe15]{detnotes}Darij Grinberg, \textit{Notes on the combinatorial
fundamentals of algebra}, 7 May 2026.\newline%
\url{http://www.cip.ifi.lmu.de/~grinberg/primes2015/sols.pdf} \newline The
numbering of theorems and formulas in this link might shift when the project
gets updated; for a \textquotedblleft frozen\textquotedblright\ version whose
numbering matches that in the citations above, see \url{https://github.com/darijgr/detnotes/releases/tag/2026-05-07}.

\bibitem[Heinig11]{Heinig}\href{http://arxiv.org/abs/1103.2717v3}{Peter
Christian Heinig, \textit{Chio Condensation and Random Sign Matrices},
arXiv:1103.2717v3}.

\bibitem[KarZha16]{KarZha16}%
\href{https://www.cip.ifi.lmu.de/~grinberg/primes2015/kazh-exp.pdf}{Karthik
Karnik, Anya Zhang, \textit{Combinatorial proof of Chio Pivotal Condensation},
25 May 2016.}

\bibitem[Sahi13]{Sahi13}\href{http://arxiv.org/abs/1309.4047v1}{Siddhartha
Sahi, \textit{Harmonic vectors and matrix tree theorems}, arXiv:1309.4047v1.}

\bibitem[Verstr12]{Verstraete}Jacques Verstraete, \textit{Math264A Lecture J},
4 December 2012.\newline\url{http://www.math.ucsd.edu/~jverstra/264A-LECTUREJ.pdf}

\bibitem[Zeilbe85]{Zeilbe}%
\href{http://www.math.rutgers.edu/~zeilberg/mamarimY/DM85.pdf}{Doron
Zeilberger, \textit{A combinatorial approach to matrix algebra}, Discrete
Mathematics 56 (1985), pp. 61--72.}\newline Re-typeset version at
\url{https://sites.math.rutgers.edu/~zeilberg/mamarimY/DM85dg.pdf} .
\end{thebibliography}
\end{document}